\documentclass[12pt,leqno,a4paper]{amsart}

\usepackage{amssymb}
\usepackage{hyperref}                         
\usepackage[all]{xy}

\usepackage[usenames]{color}

\hypersetup{pdftex,                        
pdftitle={C. Lassueur, N. Mazza: Endotrivial modules for the sporadic groups},
colorlinks=true,
citecolor=cyan,
filecolor=black,
linkcolor=blue,
urlcolor=black,
hypertexnames=true}

\newtheorem{thm}{Theorem}[section]
\newtheorem{lemma}[thm]{Lemma}
\newtheorem{cor}[thm]{Corollary}
\newtheorem{prop}[thm]{Proposition}

\theoremstyle{definition}
\newtheorem{defn}[thm]{Definition}

\def\CE{{\mathcal{E}}}

\newcommand{\spor}[1]{\operatorname{#1}}
\newcommand{\nspor}[1]{\!\operatorname{#1}}

\def\Z{{\mathbb Z}}

\def\F{{\mathbb F}}

\def\dim{\operatorname{dim}\nolimits}

\def\Res{\operatorname{Res}\nolimits}
\def\Ind{\operatorname{Ind}\nolimits}
\def\Inf{\operatorname{Inf}\nolimits}
\newcommand{\res}[2]{\!\downarrow^{#1}_{#2}}

\newcommand{\Indhg}[2]{\Ind_{#1}^{#2}}

\def\Hom{\operatorname{Hom}\nolimits}
\def\End{\operatorname{End}\nolimits}

\def\om{\Omega}

\newcommand{\ls}[2]{{^{#1}\!{#2}}}

\def\proj{\hbox{(proj)}}
\def\onto{\twoheadrightarrow}
\newcommand{\qbox}[1]{\quad\hbox{{#1}}\quad}
\newcommand{\wt}[1]{\widetilde{{#1}}}

\DeclareMathOperator{\GL}{GL} 
\DeclareMathOperator{\PSL}{L}   
\DeclareMathOperator{\PSU}{U}   
   
\DeclareMathOperator{\Syl}{Syl}    
\DeclareMathOperator{\Irr}{Irr}

\newcommand{\cO}{{\mathcal{O}}}

\newcommand{\fA}{{\mathfrak{A}}}
\newcommand{\fS}{{\mathfrak{S}}}

\newcommand{\IC}{{\mathbb{C}}}
\newcommand{\IZ}{{\mathbb{Z}}}

\newcommand{\tw}[1]{{}^#1\!}

\newcommand{\Atlas}{{\sc Atlas}}

\title{Endotrivial modules for the sporadic simple groups and their covers}
\author{Caroline Lassueur}
\address{Caroline Lassueur, FB Mathematik, TU Kaisers\-lautern, Post\-fach
  3049, 67653 Kaisers\-lautern, Germany} 
\email{lassueur@mathematik.uni-kl.de}  
\author{Nadia Mazza}
\address{Nadia Mazza, Department of Mathematics and Statistics, Lancaster
  University, Lancaster, LA1 4YF, UK}
\email{n.mazza@lancaster.ac.uk}    
\date{\today}
\subjclass[2010]{Primary 20C20; Secondary 20C34}

\thanks{The first author gratefully acknowledges financial support by ERC
  Advanced Grant 291512 and SNF Fellowship for Prospective Researchers PBELP2$_{-}$143516}

\begin{document}

\begin{abstract}
In a step towards the classification of endotrivial modules for
quasi-simple groups, we investigate endotrivial modules for the sporadic simple
groups and their covers.
A main outcome of our study is the existence of torsion endotrivial modules with dimension greater than one for several sporadic groups with $p$-rank greater than one. 
\end{abstract}

\maketitle

\pagestyle{myheadings}
\markboth{C. Lassueur \& N. Mazza}{Endotrivial modules for sporadic groups and their covers}


\section{Introduction}\label{sec:introduction}

Let $G$ be a finite group and $k$ a field of prime characteristic $p$
dividing the order of $G$. A $kG$-module $V$ is called endotrivial if
$V\otimes V^{*}\cong k\oplus Q$, with $Q$ a projective $kG$-module. The
tensor product over $k$ induces a group structure on the set of
isomorphism classes of indecomposable endotrivial $kG$-modules, called
the group of endotrivial modules and denoted $T(G)$. This group is
finitely generated and it is of particular
interest in modular representation theory as it forms an important part
of the Picard group of self-equivalences of the stable category of
finitely generated $kG$-modules. In particular the self-equivalences of
Morita type are induced by tensoring 
with endotrivial modules. 

Endotrivial modules have seen a considerable interest since defined by
Dade in 1978 in \cite{DADE1,DADE2} where he gives a classification for
abelian $p$-groups. In this case $T(G)$ is cyclic and generated by the module
$\Omega(k)$. Since then a full classification has been obtained
over $p$-groups (see \cite{CARLSON_ETSURVEY} and the references
therein). Contributions towards a general classification of endotrivial
modules  have already been obtained for  several families of general
finite groups 
(\cite{CMN1,CMN, CHM, MT, CMTpsol, CMT2,LMS,NAVROB,ROBsimple}), 
however the problem of computing $T(G)$ for an arbitrary finite group $G$
remains open in general. 
In particular the problem of determining the structure of the torsion
subgroup $TT(G)$ of $T(G)$ is the hardest part. In
\cite[Thm. A]{CARLSON_ETSURVEY} J. Carlson states that in the known
cases, but for few exceptions, $TT(G) = X(G)$, the group of
one-dimensional $kG$-modules. 

In this article, we study endotrivial modules for the sporadic simple groups and
their covering groups, i.e. perfect groups $G$ such that $G/Z(G)$ is
sporadic simple, in all characteristics $p$ dividing the order of
$G$. As a consequence we obtain a significant list of new examples where
$TT(G)$ is non-trivial, and a fortiori not isomorphic to $X(G)$. 

Our first main result is the determination of the torsion-free rank of
$T(G)$ for all covering groups $G$ of sporadic groups (see 
Proposition~\ref{prop:TFchar2} and Section~\ref{sec:TF}). Our second main
result is the full structure of $T(G)$ for all covering groups of
sporadic groups in characteristic~$2$ (see Section~\ref{sec:sporadic2}). 
Our third main result concerns
proper covering groups, i.e. cases where $Z(G)$ is non-trivial: in any
such case,  we either prove that $T(G)\cong T(G/Z(G))$ via inflation, or
compute $T(G)$ by explicit arguments (see Section~\ref{sec:sporadic2}
 and Section~\ref{sec:TT}).
Our fourth main result is the full structure of the group $T(G)$ in all
characteristics where the Sylow $p$-subgroups are cyclic (see Table~\ref{tbl:cyclic}), although in
this case the main theoretical ingredients are provided by \cite{MT}. 

At this point we are left with the computation of the torsion subgroup
$TT(G)$ in odd characteristic $p$ such that  $p^{2}$ divides the order
of $G$. In this paper we determine the full structure of $TT(G)$ in all 
characteristics for the following groups: $\spor{M}_{11}$, $\spor{M}_{12}$, 
$\spor{J}_{1}$, $\spor{M}_{22}$, $\spor{J}_{2}$, $\spor{M}_{23}$,
$\tw{2}\spor{F}_{4}(2)'$, $\spor{HS}$, $\spor{M}_{24}$,
$\spor{McL}$, $\spor{Suz}$, $\spor{Fi}_{22}$, $\spor{HN}$, $\spor{Ly}$, $\spor{J}_{4}$ and their covers.
In the remaining cases, we do not obtain the full structure of $TT(G)$ for all 
$p\,|\,|G|$, but we give as much information as possible. Typically, in most missing cases we can 
provide an upper bound in terms of abelian groups for $TT(G)$.

Our methods combine on the one hand traditional modular representation 
theory techniques such as the theory of vertices and sources 
(see \cite{CARLSON_ETSURVEY}) and on the other hand character
theory. More specifically we use the well-known fact that trivial source modules 
lift uniquely to characteristic zero and use the newly obtained characterisation 
of the torsion subgroup $TT(G)$ via ordinary characters of the group $G$  
(see Theorem~\ref{lem:Greenchars}). In particular this approach allows us to compute Green 
correspondence  at the level of ordinary characters, rather than at the level of modules, 
allowing us to make computer calculations via the GAP Character Table Libraries 
even for large sporadic groups. Whenever possible we provide the characters 
afforded by trivial source endotrivial module, so that using the known decomposition 
matrices, one can recover the composition factors of these modules. 
In addition, we rely on a recently developed approach by Balmer (\cite{Balmer}) 
which uses weak $H$-homomorphisms. Leaving aside the details, Balmer's 
method is useful in helping to find the trivial source endotrivial modules, though it
does not provide a complete answer in general.

Our paper is organised as follows. In Sections~\ref{sec:et} and
\ref{sec:sporadic}, we sum up useful results on endotrivial modules and
the sporadic groups. 
In Section~\ref{sec:sporadic2}, we determine the structure of $T(G)$
in characteristic $2$. For odd characteristics, in Section~\ref{sec:TF}
we determine the torsion-free rank of $T(G)$, and in
Section~\ref{sec:TT} the torsion subgroup $TT(G)$ when we have obtained its full
structure. We collect the results in Section~\ref{sec:tables} in
Table~\ref{tbl:sumup} if $p^{2}\,|\,|G|$ and in Table~\ref{tbl:cyclic}
if the Sylow $p$-subgroups are cyclic.\\

\textbf{Acknowledgements.}
The first author would like to thank Prof. Dr. Gunter Malle and the FB Mathematik of the TU
Kaiserslautern for hosting her as an SNF Fellow during the academic year
2013-2014, as well as the department of Mathematics of the University of
Lancaster for supporting her visit in October 2013.
The second author is thankful to the FB Mathematik of the TU
Kaiserslautern for hosting her for a short visit in Spring 2013, during
which this project started. 
Both authors are indebted to Jon Carlson for his help with
computer calculations. Finally note that the present version of the paper contains corrections to the 
proofs of Lemma 3.1(2), and Lemma 3.2(1), thanks to Gunter Malle who noticed errors.\\


\section{Endotrivial modules: background results}\label{sec:et}

Throughout, unless otherwise stated, we let $G$ be a finite group, $P$ a Sylow
$p$-subgroup of $G$, for some prime $p$ dividing the order of $G$, and
$k$ an algebraically closed field of characteristic $p$. We
write $N_G(P)$ for the normaliser of $P$ in $G$ and we adopt the same 
notation and conventions as in \cite{THEVbook}. In particular, all $kG$-modules 
are left modules and finitely generated. In addition, for a $kG$-module $M$, we 
put $M^*=\Hom_k(M,k)$ for the $k$-dual of $M$ and $\otimes=\otimes_k$. 
Also, $X(G)$ denotes the group of isomorphism classes of 1-dimensional $kG$-modules.

\subsection{Definitions and elementary properties}\label{ssec:defET}

\begin{defn}\label{def:etmodule}
A $kG$-module $M$ is {\em endotrivial} if $\End_k(M)\cong
k\oplus\proj$ as $kG$-modules, where (proj) denotes some projective
$kG$-module. 
\end{defn}

Here are the main basic properties of endotrivial modules.

\begin{lemma}\label{basics}
Let $M$ be a $kG$-module.
\begin{enumerate}
\item $M$ is endotrivial if and only if $M$ splits as the direct sum $M_0\oplus\proj$
for an indecomposable endotrivial $kG$-module $M_0$, unique up to
isomorphism. 
\item\label{item2} The relation
$$M\sim N~\Longleftrightarrow~M_0\cong N_0$$
on the class of endotrivial $kG$-modules is an equivalence relation and
every equivalence  class contains a unique indecomposable module up to
isomorphism. 
\item \label{item3}
If $M$ is endotrivial, then $M\res GH$ is endotrivial for
  any subgroup $H$ of $G$. Moreover if $H\geq P\in\Syl_{p}(G)$, then 
  $M$ is endotrivial if and only if $M\res GH$ is.
\item $M$ is endotrivial if and only if $M^*$ is endotrivial.
\item Given two endotrivial $kG$-modules $M,N$, then $M\otimes N$ is
  endotrivial. 
\item\label{item6} Let $\xymatrix{
0\ar[r]&M\ar[r]&Q\ar[r]&N\ar[r]&0}$ be a short exact sequence of
$kG$-modules with $Q$ projective. Then $M$ is endotrivial if and only if
$N$ is endotrivial.
\item If $M$ is endotrivial, then 
$$\dim_{k}(M)\equiv\left\{\begin{array}{ll} 
\pm 1\pmod{|G|_{p}}& \qbox{if $p$ is odd}\\
 \pm 1\pmod{\frac{1}{2}|G|_{2}}&\qbox{if $p=2$}
\end{array}\right.$$
where $|G|_p$ is the order of a Sylow $p$-subgroup of $G$. Also, if $M$ 
is moreover a trivial source module, then $\dim_{k}(M)\equiv 1\pmod{|G|_{p}}$. 
\end{enumerate}
\end{lemma}

For proofs of Lemma~\ref{basics}(1)-(7), we refer the reader to 
\cite{CARLSON_ETSURVEY} and the references therein.
Part~(\ref{item2}) is equivalent to saying that  two endotrivial modules
$M,N$ are equivalent if and only is they are isomorphic in the stable
module category of~$kG$. This leads us to the following definition: 

\begin{defn}\label{def:tg}
The set $T(G)$ of equivalence classes of endotrivial $kG$-modules
resulting from the relation $\sim$ of Lemma~\ref{basics}(\ref{item2})
forms an abelian group for the composition law
$$[M]+[N]:=[M\otimes N].$$
The zero element of $T(G)$ is the class~$[k]$ of the trivial module, and the additive
inverse of $[M]$ is $-[M]=[M^*]$. 
The group $T(G)$ is called the {\em group of endotrivial $kG$-modules}.
\end{defn}

If $P_*\onto M$ is a minimal projective resolution of a $kG$-module $M$,
we denote by $\Omega^n(M)$ the $(n-1)$-st kernel of the differentials, and
similarly $\Omega^{-n}(M)$ is the $(n-1)$-st cokernel in an injective
resolution $M\hookrightarrow I_*$ of $M$. By convention, $\Omega^0(M)$
is the projective-free part of $M$ and $\Omega(M)=\Omega^1(M)$. 
Thus part~(\ref{item6}) in Lemma~\ref{basics} provides us with an important
family of endotrivial modules, namely the modules $\Omega^n(k)$ for $n\in \IZ$.
Henceforth we write $\om$ for the equivalence class of $\om(k)$
in $T(G)$.  For $m,n\in\IZ$, $\Omega^{n}(k)\otimes\Omega^{m}(k)\cong
\Omega^{m+n}(k)\oplus \proj$, hence $\left <\Omega\right>\leq T(G)$ is a
cyclic subgroup, and, unless $G$ has $p$-rank~one, $\left
  <\Omega\right>\cong \IZ$. 

\subsection{Previously known results}\label{ssec:prevresults}

The abelian group $T(G)$ is known to be finitely generated \cite[Cor. 2.5]{CMN1} and so we can
write 
$$T(G)=TT(G)\oplus TF(G)$$
where $TT(G)$ is the torsion subgroup of $T(G)$ and $TF(G)$ is a torsion-free direct 
sum complement of $TT(G)$ in
$T(G)$.  In particular $TT(G)$ is finite. 
The rank of $TF(G)$ is the {\em torsion-free rank} of $T(G)$ and it is
written $n_G$. It is determined in~\cite{CMN1}.

\begin{lemma}[\cite{CMN1}]\label{lem:ranktf}
The torsion-free rank of $T(G)$ is equal to the number of
conjugacy classes of maximal elementary abelian $p$-subgroups of rank
$2$ if $G$ has $p$-rank $2$, or that number plus one if $G$ has $p$-rank
greater than $2$. 
\end{lemma}


The next three results characterising the torsion-free rank of $T(G)$ are at the crossroad 
of local group structure and representation theory. 

\begin{thm}[\cite{GM,MacW}]\label{thm:connected}
Let $G$ be a finite $p$-group. Suppose that $G$ has a maximal elementary
abelian $p$-subgroup of order $p^2$. Then $G$ has $p$-rank at most $p$
if $p$ is odd, or at most $4$ if $p=2$.
\end{thm}
For completeness, the case $p=2$ is due to A. MacWilliams (often coined
the {\em four-generator theorem}), and for odd
primes it is work of G. Glauberman and the second author.

\begin{cor}\label{cor:torsionfreerank1}
Assume the $p$-rank of the group $G$ is greater than $p$ if $p$ is odd, 
or greater than $4$ if $p=2$. Then $T(G)$ has torsion-free rank one.
\end{cor}

\begin{thm}[\cite{CARLSONelab,MAZposet}]\label{thm:max-tf}
Let $G$ be a finite $p$-group. Then $n_G\leq p+1$ if $p$ is odd, or at
most $n_G\leq 5$ if $p=2$. Moreover, both upper bounds are optimal.
\end{thm}

Optimality of the bound $n_G$ for $p$ odd is for instance obtained with
$G$ an extraspecial $p$-group of order $p^3$ and exponent $p$ if $p$
is odd, whereas for $p=2$, we can take for $G$ an extraspecial
$2$-group of order $32$ of the form $Q_8*D_8$.\\

We now present a collection of useful results, which have been proven in previous 
articles on endotrivial modules, see \cite{CARLSON_ETSURVEY} and the 
references therein.
Recall that a subgroup $H$ of $G$ is {\em strongly $p$-embedded} if $p$ divides 
the order of $H$ but $p$ does
not divide the order of $H\cap\ls gH$ for any $g\in G\setminus H$. In
particular, $N_G(P)$ is strongly $p$-embedded in $G$ if and only if the
Sylow $p$-subgroup $P$ is a trivial intersection (T.I. hereafter) subset of
$G$. Also, if $H$ is strongly $p$-embedded in $G$, then $H\geq N_G(P)$.

Useful information on $T(G)$ can sometimes be gathered by comparison
with known endotrivial modules for subgroups of $G$. For instance, if
$H$ is a subgroup of $G$, the restriction along the inclusion induces a
group homomorphism $\Res^G_H~:~T(G)\to T(H)$. 
In contrast, the induction $\Ind_H^G(M)$ of an endotrivial $kH$-module $M$ 
needs not be endotrivial for $G$.

\begin{lemma}[Omnibus Lemma]\label{lem:omnibus}
Let $H$ be a subgroup of $G$ containing the normaliser $N_G(P)$ of a
Sylow $p$-subgroup $P$ of $G$.
\begin{enumerate}
\item If $G$ has $p$-rank at least 2, then $\Omega$ generates an infinite cyclic 
direct summand of $TF(G)$ and  can be chosen as part of a set of generators for $TF(G)$.
In particular, if $T(G)$ has torsion-free rank $1$, then $TF(G)=\langle\Omega\rangle$. 
\item\label{omni2} 
The restriction map~$~\Res^G_H:~T(G)\longrightarrow~T(H)$ is  injective.
More precisely, if $M$ is an indecomposable endotrivial $kG$-module
and if $M\res GH = L\oplus\proj$ where $L$ is an indecomposable
$kH$-module, then $M$ is the $kG$-Green  correspondent of~$L$ and $L$ is
endotrivial.  
\item Assume that for all $x\in G$ the subgroup $\ls xP\cap P$ is non
trivial. Then, the kernel of the restriction map $\Res^G_P:T(G)\to T(P) $ is
generated by the isomorphism classes of $1$-dimensional $kG$-modules. In
particular, this holds whenever $G$ has a non-trivial normal $p$-subgroup.
\item If $H$ is strongly $p$-embedded in $G$, then $\Res^G_H:T (G)\longrightarrow T(H)$ 
is an isomorphism, the inverse
map being induced by the induction $\Ind_H^G$.
\item 
If $P=N_G(P)$ is neither cyclic, nor semi-dihedral, nor
  generalised quaternion, then $TT(G)=\{[k]\}$. 
\end{enumerate}
\end{lemma}

Lemma~\ref{lem:omnibus} shows that a large part of $TT(G)$ is generated
by the (stable) isomorphism classes of trivial source endotrivial
modules. Indeed, unless a Sylow $p$-subgroup $P$ of $G$ has normal
$p$-rank one, then $T(P)$ is torsion-free and so $TT(G)$ is detected by
restriction to $P$ in the sense of Lemma~\ref{lem:omnibus}(\ref{omni2}).\\

Many of the Sylow $p$-subgroups of sporadic groups are cyclic for $p>2$.
In fact if $G$ is quasi-simple with $G/Z(G)$ sporadic simple and $P\in
\Syl_{p}(G)$ is cyclic, then $P\cong C_{p}$ and $p\neq 2$. 
So \cite[Theorem 3.2]{MT} describes the structure of $T(G)$ in this case.

\begin{thm}[\cite{MT}]\label{thm:cyclic}
Let $G$ be a finite group with a non-trivial cyclic Sylow $p$-subgroup $P$ with $|P|>2$.
Let $Z$ be the unique subgroup of $P$ of order $p$ and let $H=N_G(Z)$. Then:
\begin{enumerate}
 \item[(1)]   $T(G)=\{\; [\Ind_H^G(M)] \,\mid\, [M]\in T(H) \;\}\cong T(H)\,.$
  \item[(2)] There is an exact sequence
$$\xymatrix{0\ar[r]&X(H)\ar[r]&T(H)\ar[r]^{\Res^H_P}&T(P)\ar[r]&0}$$
where $T(P)\cong \IZ/2$. This sequence splits if and only if $[\Omega^{2}_{H}(k)]$ is a square in $X(H)$. 
In particular, this sequence splits if the index $e:=|N_{G}(Z):C_{G}(Z)|$ is odd.
\end{enumerate}  
\end{thm}

\subsection{Character theory for endotrivial modules}\label{ssec:ETchars}
Finally we recall  known results needed for our investigations which
stem from ordinary character theory. 

\begin{thm}[{}{\cite[Theorem 1.3 and Corollary 2.3]{LMS}}]\label{thm:lift}
Let $(K,\cO,k)$ be a splitting $p$-modular system. Let $V$ be an endotrivial $kG$-module. Then
\begin{enumerate}
\item[(1)] $V$ is liftable to an endotrivial $\cO G$-lattice.
\item[(2)] Moreover, if $V$ is liftable to a $\IC G$-module affording the
  character $\chi$, then $|\chi(g)|=1$  for every $p$-singular element
  $g\in G$. 
\end{enumerate}
\end{thm}

Further, if a $kG$-module $M$ is the $kG$-Green correspondent of a
$1$-dimensional $kN_{G}(P)$-module, then $M$ is a trivial source module
and it lifts uniquely to a trivial source  $\cO G$-lattice $\hat
M$. Denote by $\chi_{\hat M}$ the ordinary character afforded by $\hat
M$. Then endotriviality for $M$ can be read from the character table of
$G$ as follows: 

\begin{thm}[{}{\cite[Theorem 2.2]{LM}}]\label{lem:Greenchars}
Let $P\in Syl_{p}(G)$ and let $M$ be the $kG$-Green correspondent of a
1-dimensional $kN_{G}(P)$-module. Then $M$ is endotrivial if and only
if $\chi_{\hat M}(x)=1$  for all $p$-elements $x\in G\setminus\{1\}$. 
\end{thm}

The latter proposition relies on the following result by Green, Landrock and Scott
on character values of trivial source modules, which we will also implicitly 
extensively use in computations in Sections~\ref{sec:sporadic2} and~\ref{sec:TT}. 

\begin{lemma}[{}{\cite[Part II, Lemma~12.6]{Lan}}]\label{lem:LandScott}
 Let $M$ be an indecomposable trivial source $kG$-module and $x\in G$ a
 $p$-element.  Then,
 \begin{enumerate}
  \item[\rm(1)] $\chi_{\hat M}(x)\geq 0$ is an integer (corresponding to the
   multiplicity of the trivial $k\!\left<x\right>$-module as a direct summand
   of $M\res{G}{\left<x\right>}$);
  \item[\rm(2)] $\chi_{\hat M}(x)\neq 0$ if and only if $x$ belongs to a vertex
   of $M$.
 \end{enumerate}
\end{lemma}

\section{Sporadic groups and their covers}\label{sec:sporadic}

\subsection{Notation and terminology}\label{ssec:nota}
We call \textit{sporadic simple} any of the 26 standard sporadic simple
groups together with the Tits simple group $\tw{2}{\,\spor F}_4(2)'$.
We use the $\Atlas$~\cite{ATLAS} notation, and in particular:
\begin{itemize}
  \item[$A\times B$] denotes a direct product of groups $A$ and $B$;
  \item[$A.B$] or $AB$ denotes a group having a normal subgroup
    isomorphic to $A$ with corresponding quotient isomorphic to $B$; 
   \item[$A:B$] denotes a split extension of $A$ by $B$; 
   \item[$A\ast B$] denotes a central product of $A$ by $B$; 
  \item[$m$] denotes a cyclic group of order $m$;
  \item[$p^{n}$] denotes an elementary abelian $p$-group of order $p^{n}$; 
  \item[$p_{+}^{1+2n}$] where $p$ is an odd prime, denotes an
    extraspecial group of order $p^{1+2n}$ and {exponent $p$};
  \item[$D_{2^{n}}$] denotes  a dihedral 2-group of order $2^{n}$;
 \item[$SD_{2^{n}}$] denotes a semi-dihedral 2-group of order $2^{n}$;
 \item[$Q_{8}$] denotes the quaternion 2-group of order~8;
 \item[$\fS_{n}$] denotes the symmetric group on $n$ letters;
 \item[$\fA_{n}$] denotes the alternating group on $n$ letters.
\end{itemize}
 If more than one symbol is used, we read them left to right,
    that is we write $A.B:C$ for $(A.B):C$. 
In addition $Z(G)$ denotes the centre of $G$ and $G'=[G,G]$ the
commutator subgroup of $G$. Also, a $p'$-group for a
    prime $p$ is a group of order not divisible by $p$.

\subsection{Covering groups and inflation}\label{ssec:covgrps}

A finite group $G$ is called {\em quasi-simple} if $G$ is a perfect central
extension of a simple group. If $G$ is quasi-simple, then $H :=
G/Z(G)$ is simple and $G$ is a perfect central extension of $H$. We then
call $G$ a {\em covering group of $H$}. According to
Notation~\ref{ssec:nota}, $m.H$ denotes a covering group of a group $H$,
which is a central extension of $H$ by a cyclic group of order $m$. The
integer $m$ is called the {\em degree} of the central extension.
If $\ell$ is a prime dividing $m$,
then there exists a covering group $K$ of $H$ such that $m.H$ is a
central extension of $K$ of degree $\ell$.   

\begin{lemma}\label{lem:cext}
Let $\xymatrix{1\ar[r]&A\ar[r]&G\ar[r]^-{\pi}&H\ar[r]&1}$ be a central
extension of $H$ by $A$, where $G$ is a perfect group, and let  
$P,Q$ be Sylow $p$-subgroups of $G$ and $H$
respectively. Then the following hold.
\begin{itemize}
  \item[(1)]  $X(N_H(Q))$ is isomorphic to a quotient group of
    $X(N_G(P))$. 
  \item[(2)] If $p$ divides $|A|$ and $TT(P)=\{[k]\}$, then
    $TT(G)=\{[k]\}$.
  \item[(3)] If $p$ does not divide $|A|$, then the torsion-free rank of
    $T(H)$ equals the torsion-free rank of $T(G)$. 
\end{itemize}  
\end{lemma}
Recall that $T(P)$ is torsion-free if and only if $P$ is neither cyclic,
generalised quaternion nor semi-dihedral. 

\begin{proof} \vbox{\  }
\begin{itemize}
\item[(1)] By assumption, $\pi$ induces surjective homomorphisms 
$\xymatrix{N_G(P)\ar[r]^\pi&N_H(Q)}$ with kernel $A$, and 
$\xymatrix{PN_G(P)'\ar[r]^\pi&QN_H(Q)'}$ with kernel $O_{p}(A)(A\cap
N_G(P)')$. So $\pi$ induces a surjective homomorphism
$N_G(P)/PN_G(P)'\twoheadrightarrow N_H(Q)/QN_H(Q)'$ with kernel
isomorphic to $A/O_p(A)(A\cap N_G(P)')$. Therefore $X(N_H(Q))$ is
isomorphic to a quotient of $X(N_G(P))$.
  \item[(2)] Since $T(P)$ is torsion-free, then $TT(G)$ is generated by
    the classes of the trivial source endotrivial modules. 
Now since $p\,|\,|A|$,  Lemma~\ref{lem:omnibus}(3)  implies that such module must have dimension $1$,
but because $G$ is perfect, only the trivial module has dimension $1$.
  \item[(3)] The torsion-free ranks of $T(G)$ and $T(H)$ are determined
    by the number of conjugacy classes of maximal elementary abelian
    subgroups of rank $2$. The claim follows from the fact that  $A$ is
    central in $G=A.H$.
\end{itemize}  
\end{proof}

When investigating endotrivial modules we often rely on an operation called inflation.
Namely, if $N$ is a normal subgroup of the group $G$, the restriction along the
natural projection map $G\rightarrow G/N$ makes a $k[G/N]$-module $V$
into a $kG$-module $\Inf_{G/N}^G(V)$ on which $N$ acts trivially. This operation is 
the {\em inflation} from $G/N$ to $G$. If $p\,\nmid |N|$, then inflation
of an endotrivial $k[G/N]$-module is an endotrivial
$kG$-module. Moreover, if $N$ is central, independently from the order
of $N$, inflation commutes with Green correspondence in the following sense.

\begin{lemma}\label{lem:inf}
Let $G$ be a finite group and let $N$ be a normal subgroup of $G$.
\begin{itemize}
\item[(1)]  Let $M$ be an endotrivial $k[G/N]$-module with $\dim_{k}(M)>1$. Then
  $\Inf_{G/N}^G(M)$ is an endotrivial $kG$-module if and only if $p$
  does not divide $|N|$.
In particular, if $p$ does not divide $|N|$, then inflation induces a 
well-defined group homomorphism
  $${\Inf_{G/N}^G:T(G/N)\to T(G):}[M]\mapsto[\Inf_{G/N}^{G}(M)]\,.$$ 
\item[(2)] Assume that $N\leq Z(G)$ and let $P\in\Syl_{p}(G)$. Let $V$
  be a $kG$-module with vertex $P$ such that
  $V=\Inf_{G/N}^{G}(\overline{V})$ for some $k[G/N]$-module
  $\overline{V}$.  Let $\Gamma(-)$ and $\overline{\Gamma}(-)$ denote the
  Green correspondence from $G$ to $N_G(P)$ and from $G/N$ to
  $N_{G}(P)/N$ respectively.  Then
  $\Inf_{N_{G}(P)/N}^{N_{G}(P)}(\overline{\Gamma}(\overline{V}))\cong
 \Gamma(\Inf_{G/N}^{G}(\overline{V}))$. 
\end{itemize}  
\end{lemma}

\begin{proof} \vbox{\  }
\begin{itemize}
\item[(1)] 
Set $H:=G/N$ and let $P,Q$ be Sylow $p$-subgroups of $G$ and $H$ respectively.
If $p$ divides $|N|$, then non-trivial $p$-subgroups of $N$ acts trivially on $\Inf_H^G(L)$ for every 
projective $kH$-module $L$. Therefore it follows by definition that the inflation 
$\Inf_H^G(V)$ of an endotrivial
$kH$-module $V$ is not an endotrivial $kG$-module. Now suppose that $p$ does 
not divide $|N|$, choose an isomorphism $\phi:P\to Q$ and denote  by $\Res_\phi$  
the restriction along $\phi$ of a $kQ$-module. 
Then, because  $\Res^G_P\Inf_H^G(M)\cong\Res_\phi\Res^H_Q(M)$ for any 
$kH$-module $M$, by Lemma~\ref{basics}(\ref{item3}) inflation 
preserves endotrivial modules. The claim follows.
\item[(2)] See \cite[Proposition~2.9]{DANZKUEL}.
\end{itemize}  
\end{proof}

\subsection{T.I. Sylow $p$-subgroups}\label{ssec:TI}

We end this section with the list of the finite simple groups
which have noncyclic trivial intersection (T.I.) Sylow $p$-subgroups. 
Recall that $P\subset G$ is T.I. if $P\cap \ls gP=\{1\}$ for
any $g\in G\setminus N_G(P)$. In particular, if $|P|$ is prime, then $P$
is T.I.. 

\begin{thm}\cite[Proposition~4.6]{MICHLER}\label{thm:ti}
Let $G$ be a nonabelian simple group with a noncyclic T.I. Sylow
$p$-subgroup $P$. Then $G$ is isomorphic to one of the following groups.
\begin{enumerate}
\item $\PSL_{2}(q)$ where $q=p^n$ with $n\geq2$.
\item $\PSU_{3}({q})$ where $q$ is a power of $p$.
\item $\ls2{\,\spor B}_2(2^{2m+1})$ and $p=2$.
\item $\ls2{\,\spor G}_2(3^{2m+1})$ with $m\geq1$ and $p=3$.
\item $\PSL_{3}(4)$ or $\spor M_{11}$ and $p=3$.
\item $\ls2{\,\spor F}_4(2)'$ or $\spor{McL}$ and $p=5$.
\item $\spor J_4$ and $p=11$.
\end{enumerate}
\end{thm}

\smallskip

We are now ready to combine both Sections~\ref{sec:et}
and~\ref{sec:sporadic} in order to investigate endotrivial modules for
the sporadic groups and their covers. 
We start with the case $p=2$ and whenever we think it more appropriate, we
use the algebra softwares MAGMA~\cite{MAGMA} or GAP~\cite{GAP4,CTblLib}. For
the clarity of exposition, we summarise the results in 
Table~\ref{tbl:sumup} and Table~\ref{tbl:cyclic} in Section~\ref{sec:tables}.

\section{Endotrivial modules for sporadic groups in characteristic
  $2$}\label{sec:sporadic2} 

In this section, we prove the results stated in Section~\ref{sec:tables}
for a field $k$ of characteristic $2$. 
The $2$-local structure of simple groups has  been thoroughly analysed.
 In particular, by a celebrated result due to
Brauer and Suzuki it is well-known that a Sylow $2$-subgroup of a simple
group cannot be generalised quaternion (\cite{BS59}), while
\cite[Corollary~2, p.~144]{Suz86} asserts that a Sylow $2$-subgroup of a
nonabelian simple group cannot be cyclic. In view of these results it
seems reasonable to believe that if $G$ is a nonabelian simple group
then $T(G)$ has torsion-free rank one in characteristic $2$. We can prove
slightly better.

\begin{prop}\label{prop:TFchar2}
Let $G$ be a finite quasi-simple group with $G/Z(G)$ sporadic simple,
then the torsion-free rank of $T(G)$ is one, that is,
$TF(G)=\langle\om\rangle\cong\IZ$. 
\end{prop}

\begin{proof}
Using Theorems~\ref{thm:connected} and~\ref{thm:max-tf}, we need to show
that if $G$ has a nonabelian Sylow $2$-subgroup $P$ of rank $2,3$ or
$4$, then $P$ is either semi-dihedral, or $P$ has rank greater than $2$
and no maximal elementary abelian subgroup of order $4$. 
By inspection of \cite[Table 5.6.1]{GLS3}, we are left with the
following groups.
If the $2$-rank is $2$, then $G=\spor M_{11}$ and has a semi-dihedral
Sylow $2$-subgroup of order $16$. It follows from~\cite[Thm.~6.5]{CMT2}
that $n_{G}=1$. 
Otherwise $G$ is one of $\spor{M}_{12},2.\nspor{M}_{12},
\spor{M}_{22}, 4.\nspor{M}_{22},3.\nspor{M}_{22}, 12.\nspor{M}_{22},
\spor{M}_{23}, \spor{J}_{1}, \spor{J}_{2}, 2.\nspor{J}_{2},\\ J_{3},
3.\nspor{J}_{3}, \spor{Co}_{3}, \spor{HS}, \spor{McL},
3.\nspor{McL},\spor{Ly},\spor{O'N},3.\nspor{O'N}$ and its $2$-rank is
$3$ or $4$. A routine verification of each of these cases shows that a
Sylow $2$-subgroup of $G$ has no maximal elementary abelian subgroup of
order $4$.
\end{proof}

By Lemma~\ref{lem:omnibus} and the specific $2$-local structure of
sporadic simple groups and their covers, the torsion subgroup of $T(P)$
is often trivial, and the Sylow $2$-subgroups are often self-normalising. This
leads us to our next result.

\begin{lemma}\label{lem:sn2sylow}
If $G$ is one of the groups $\spor M_{12}$,
$2.\nspor{M}_{12}$, $\spor M_{22}$, $2.\nspor{M}_{22}$,
$4.\nspor{M}_{22}$, $6.\nspor{M}_{22}$, $12.\nspor{M}_{22}$,
$2.\spor{J}_{2}$, $\spor M_{23}$, $\spor M_{24}$, $\spor{HS}$,
$2.\nspor{HS}$,  $\spor{McL}$, $\spor{He}$, $\spor{Ru}$, $2.\nspor{Ru}$,
$\spor{O'N}$, $\spor{Co}_{3}$, $\spor{Co}_{2}$, $\spor{Fi}_{22}$,
$2.\nspor{Fi}_{22}$, $6.\nspor{Fi}_{22}$, $\spor{Ly}$, $\spor{Th}$,
$\spor{Fi}_{23}$, $\spor{Co}_{1}$, $2.\nspor{Co}_{1}$, $\spor J_{4}$,
$\spor{Fi}'_{24}$, $\spor B$, $2.\nspor{B}$, $\spor M$, or
$\tw{2}\spor{F}_{4}(2)'$, then $TT(G)=\{[k]\}$. 
\end{lemma}

\begin{proof}
Let $P\in\Syl_{2}(G)$. 
If $G$ is one of  $\spor M_{12}$, $2.\nspor{M}_{12}$,
$\spor M_{22}$, $2.\nspor{M}_{22}$, $4.\nspor{M}_{22}$, $\spor M_{23}$,
$\spor M_{24}$, $\spor{HS}$, $2.\nspor{HS}$,  $\spor{McL}$, $\spor{He}$,
$\spor{Ru}$, $2.\nspor{Ru}$, $2.\nspor{Suz}$, $6.\nspor{Suz}$,
$\spor{O'N}$, $\spor{Co}_{3}$, $\spor{Co}_{2}$, $\spor{Fi}_{22}$,
$2.\nspor{Fi}_{22}$, $6.\nspor{Fi}_{22}$, $\spor{Ly}$, $\spor{Th}$,
$\spor{Fi}_{23}$, $\spor{Co}_{1}$, $2.\nspor{Co}_{1}$, $\spor J_{4}$,
$\spor{Fi}'_{24}$, $\spor B$, $2.\nspor{B}$, $\spor M$, or
$\tw{2}\spor{F}_{4}(2)'$, then one can read from \cite[Table 1]{AnWil}
that $P$ is self-normalising.  
Therefore, $TT(G)=\{[k]\}$ by Lemma~\ref{lem:omnibus}. If $G$ is one of
$6.\nspor{M}_{22}$, $12.\nspor{M}_{22}$, $2.\spor{J}_{2}$,
$2.\nspor{Suz}$, $6.\nspor{Suz}$, or $6.\nspor{Fi}_{22}$ then $G$ has a
non-trivial normal $2$-subgroup and so $TT(G)=\{[k]\}$ by
Lemma~\ref{lem:cext}.  
\end{proof}

\begin{lemma}\label{lem:coverschar2}
If $G$ is one of the $3$-fold covers $3.\nspor{M}_{22}$,
$3.\nspor{Suz}$, $3.\nspor{O'N}$,  $3.\spor{Fi}_{22}$, or
$3.\spor{Fi}'_{24}$, then $G$ has no faithful endotrivial
module. Therefore $T(G)\cong T(G/Z(G))\cong \IZ$ via inflation.  
\end{lemma}

\begin{proof}
Assume $M$ is a faithful endotrivial $kG$-module and let $\chi$ be the 
ordinary character of $G$ belonging to a lift of $M$. Then all
irreducible constituents of $\chi$ lie in a common faithful block of
$kG$ and by Theorem~\ref{thm:lift} the values of $\chi$ on $2$-singular
conjugacy classes are roots of unity. However, it can be read from the
GAP Character Table Libraries \cite{CTblLib} (or the $\Atlas$
\cite{ATLAS}) that $G$ has a $2$-singular class $C$ such that for all
faithful characters $\psi\in\Irr(G)$ we have $\psi(c)\in m\IZ$ with
$m\in \IZ\setminus\{\pm 1\}$, $c\in C$, hence a contradiction. 
 More precisely for $G=3.\nspor{M}_{22}$, take $(m,C)=(2,6c)$,
 for $G=3.\nspor{Suz}$ take $(m,C)=(0,6q)$, for $G=3.\nspor{O'N}$
 take $(m,C)=(0,6c)$, for $G=3.\spor{Fi}_{22}$ take
 $(m,C)=(0,6w)$, for $G=3.\spor{Fi}'_{24}$ take
 $(m,C)=(0,6ad)$. (Notation for conjugacy classes is that of GAP
 \cite{CTblLib}.) 
\end{proof}

We are left with the computation of the torsion subgroup $TT(G)$ in
characteristic $2$ for the following groups: 
$\spor M_{11}$, $\spor{J}_{1}$, $\spor{J}_{2}$,
$\spor{J}_{3}$, $3.\nspor{J}_{3}$, $3.\nspor{McL}$, $\spor{Suz}$,
$\spor{HN}$. We proceed case by case. 

\subsection{The Mathieu group $\spor M_{11}$}\label{sec:mathieu2}\

The group $G=\spor M_{11}$ has a semi-dihedral Sylow $2$-subgroup $P$ of
order $16$ and $P=N_G(P)$. From~\cite[Theorems~6.4 and~6.5]{CMT2}, we
know that $T(\spor M_{11})=\langle\om,[M]\rangle\cong\Z\oplus\Z/2$
where $M$ is the $kG$-Green correspondent of the translated relative
syzygy $\om(\om_{P/C}(k))$ for a noncentral subgroup $C$ of $P$ of order
$2$. Moreover $M$ is self-dual. This module was explicitly described by Kawata and 
Okuyama  \cite[Thm. 5.1]{KAW} as the relative syzygy
$M\cong\om(\om_{G/C}(k))$. Computations with MAGMA show that $\dim_{k}(M)=41$.

\subsection{The Janko groups: $\spor{J}_{1}$, $\spor{J}_{2}$,
$\spor{J}_{3}$ and $3.\nspor{J}_{3}$}\label{sec:janko2}\

Let $G=\spor J_{1}$ and $P$ a Sylow $2$-subgroup of $G$. 
Then $N_G(P)\cong P:7:3$ (cf \cite{ATLAS}) and $X(N_G(P))\cong \IZ/3$ is
generated by a $1$-dimensional $kN$-module $\lambda$ with
$\lambda^{\otimes3}=k$. 
The socle and Loewy series of the $kG$-Green correspondent
$\Gamma(\lambda)$ of $\lambda$ have been computed in \cite[Cor
5.8]{LaMi}, from which we gather that $\Gamma(\lambda)$ has dimension
$133$ and affords the ordinary character $\chi_{12}$. By Lemma~\ref{basics}(7),
$\Gamma(\lambda)$ cannot be endotrivial as $133\not\equiv 1\pmod{|P|}$. 
Therefore $TT(\spor J_1)=\{[k]\}$.

The groups $\spor J_2$ and $\spor J_3$ have very similar $2$-local
structure. Put $G$ for any of these Janko groups.
The group $G$ has a Sylow $2$-subgroup $P=(Q_8*D_8):2^2$ of order $2^7$
and $2$-rank $4$. The index of $P$ in its normaliser $N$ is $3$, so that
$X(N)\cong \IZ/3$. Moreover $G$ has a maximal subgroup $H$ containing $N$
and $H$ is isomorphic to $(Q_8*D_8):\fA_5$.
The permutation module $k[H/P]$ gives all the indecomposable
trivial source $kH$-modules up to isomorphism and MAGMA computations
give that $k[H/P]$ splits into a
direct sum of modules with dimensions $1,4,5,5$.
Therefore $k$ is the only torsion endotrivial $kH$-module and a fortiori
$TT(G)=\{[k]\}$ for $G=\spor J_2$ and $G=\spor J_3$.
For $G=3.\spor J_3$ we read from \cite{ATLAS} that $G$ has a maximal
subgroup $C(2A)$ which is the centraliser of the central involution of
$P$. Using MAGMA, we obtain that
$C(2A)\cong3\times(Q_8*D_8).\fA_5$ and its derived subgroup is
$(Q_8*D_8).\fA_5$. Therefore $X(C(2A))\cong\IZ/3$ and
$TT(C(2A))\cong\Z/3$. Referring the reader to the method developed in
\cite{Balmer}, as explained in Section~\ref{sec:TT}, a
further MAGMA computation shows that $TT(3.\nspor{J}_{3})\cong\IZ/3$. In
other words the Green correspondents of the two non-trivial
$1$-dimensional $k[C(2A)]$-modules are endotrivial. 

\subsection{The group $3.\nspor{McL}$} \label{sec:3mcl2}\
Let $G=3.\nspor{McL}$ and $P$ a Sylow $2$-subgroup of $G$. Then $|P|=2^7$
and we have $X(N_G(P))\cong \IZ/3$ since a Sylow $2$-subgroup of
$\spor{McL}$ is self-normalising.
Because $TT(\spor{McL})=\{[k]\}$ by Lemma~\ref{lem:sn2sylow}, 
a non-trivial trivial source endotrivial $kG$-module must be faithful.
We see from the GAP Character Table Libraries \cite{CTblLib} that
$G$ has three $2$-blocks of full defect: the principal block $B_{0}$ and
two faithful blocks $B_{7}$ and $B_{8}$. 
However all simple modules in $B_{7}$ and $B_{8}$ have even dimension,
so that by Lemma~\ref{basics}(7) none of these blocks can contain
endotrivial modules, whence $TT(3.\nspor{McL})=\{[k]\}$.

\subsection{The group $\spor{Suz}$}\label{sec:suz2}\

The group $\spor{Suz}$ has a Sylow $2$-subgroup $P$ of order $2^{13}$,
rank $6$ and index $3$ in its normaliser $N$, so that $X(N)\cong \IZ/3$. Furthermore $N$ is
contained in a maximal subgroup $H$ of $\spor{Suz}$ of the form
$H = 2^{1+6}_{-}.\spor U_4(2)$ and $|H:N|=135$. Let $1_{10}\in X(N)$ be 
the 10-th linear character of $N$ (according to the GAP character tables \cite{CTblLib} notation), 
which has order $3$. Then 
$$\Ind_{N}^{H}(1_{10})=\chi_{8}+\chi_{12}+\chi_{13}+\chi_{18}\,,$$
where $\chi_{8}, \chi_{12}, \chi_{13}, \chi_{18}\in\Irr(H)$ have degree $15,30,30$ 
and $60$ respectively. Since all four of these characters have classes $2a$ 
and $2b$ in their kernels, the $kH$-Green correspondent $\Gamma_{H}(1_{10})$ 
cannot be endotrivial by Theorem~\ref{lem:Greenchars}. Therefore, neither is the $kG$-Green correspondent $\Gamma_{G}(1_{10})$ and we obtain $TT(\spor{Suz})=\{[k]\}$.

\subsection{The group $\spor{HN}$ (also denoted $\spor
  F_5$ or $\spor
  F_{5+}$ )}\label{sec:hn2}\ 

The group $G=\spor{HN}$ has a Sylow $2$-subgroup $P$ of order
$2^{14}$ and $G$ has two maximal subgroups which are
centralisers of involutions. We gather from \cite{ATLAS} that the
centraliser of a 2-central involution has the form
$H = 2^{1+8}_+.(\fA_5\times \fA_5).2$.
So $H$ must contain $N_G(P)$ and we have $TT(H)=X(H)$ since
$H$ has a non-trivial normal $2$-subgroup. Now $X(H)=H/PH'=1$ because
$H'\geq (\fA_5\times \fA_5)'=\fA_5\times \fA_5$. Therefore, by injectivity of
the restriction map $T(G)\to T(H)$ we conclude that $TT(G)=\{[k]\}$.


\section{The torsion-free rank of $TF(G)$ in odd characteristic and
  $p$-rank $\geq 2$}\label{sec:TF} 

Throughout this section, we assume that $p$ is odd. We compute the
torsion-free rank $n_{G}$ of $T(G)$, for $G$ a quasi-simple group such
that $G/Z(G)$ is sporadic simple and has a noncyclic Sylow
$p$-subgroup. Recall from Lemma~\ref{lem:ranktf} that $n_G$ is equal to
the number of conjugacy classes of maximal elementary abelian
$p$-subgroups of rank $2$ if $G$ has $p$-rank~$2$, or that number plus
one if $G$ has $p$-rank greater than $2$. 

\subsection{Sporadic groups with $p$-rank $2$}\label{sec:prank2}
First we assume that the $p$-rank is exactly $2$. If $P\in\Syl_{p}(G)$,
then a Sylow $p$-subgroup $P$ of $G$ is either elementary abelian of rank $2$, or
$P$ is an extraspecial $p$-group of order $p^3$ and exponent $p$ (see
\cite{GLS3}).

Suppose that $P$ is elementary abelian. Then $n_{G}=1$ and
$TF(G)=\langle\Omega\rangle\cong \IZ$. This occurs in the following cases:
$$\begin{aligned}\label{al:2spor}
\hbox{characteristic $3$,}\quad&\spor M_{11}\;,\;\spor
M_{22}\;,\;2.\nspor M_{22}\;,\;4.\nspor M_{22}\;,\;\spor
M_{23}\;,\;\spor{HS}\;,\;2.\spor{HS},\\
\hbox{characteristic $5$,}\quad&\spor J_2\;,\;2.\spor
J_2\;,\;\spor{He}\;,\;\spor{Suz}\;,\;2.\nspor{Suz}\;,\;
3.\nspor{Suz}\;,\;6.\nspor{Suz}\;,\;\\
&\spor{Fi}_{22}\;,\;2.\nspor{Fi}_{22}\;,\;
3.\nspor{Fi}_{22}\;,\;
6.\nspor{Fi}_{22}\;,\;\spor{Fi}_{23}\;,\;\spor{Fi}'_{24}\;,\;3.\spor{Fi}'_{24}\;,\;\tw2\spor F_4(2)',\\
\hbox{characteristic $7$,}\quad&\spor{Co}_{1}\;,\;
2.\nspor{Co}_{1}\;,\;\spor{Th}\;,\;\spor B\;,\;2.\spor B,\\
\hbox{characteristic $11$,}\quad&\spor M.
\end{aligned}$$

Suppose now that $P$ is extraspecial of order $p^3$ and exponent $p$, that is $P=p_{+}^{1+2}$ according to
Notation~\ref{ssec:nota}. We
use the presentation 
$$P=\langle x,y,z\,|\, x^{p}=y^{p}=z^{p}=[x,z]=[y,z]=1, [x,y]=z \rangle$$
where $[x,y]=xyx^{-1}y^{-1}$. Then $Z(P)=P'=\Phi(P)=\langle z\rangle$ 
has order $p$, where $\Phi(P)$
is the Frattini subgroup. The outer automorphism group 
of $P$ is isomorphic to $\GL_{2}(p)$ and the matrix 
$\bigl(\begin{smallmatrix}r' & r \\s' & s\end{smallmatrix}\bigr)$ with
determinant $d$ corresponds to the class of automorphisms sending $x$ to
$x^{r'}y^{s'}$, $y$ to $x^{r}y^{s}$ and $z$ to $z^{d}$.

Ruiz and Viruel \cite{RV} have classified the saturated fusion systems
on $P$ providing us with the information we need in order to compute
$n_{G}$. That is, the number of $G$-conjugacy classes of
the $p+1$ elementary abelian $p$-subgroups of rank~$2$. 
We refer the reader to \cite{RV} for details, and only give the results
that apply in our context specifically.   We recall  that a $p$-subgroup $Q$ is
{\em $p$-centric} in $G$ if $Z(Q)$ is a Sylow $p$-subgroup of $C_G(Q)$,
and $Q$ is {\em $p$-radical} if $N_G(Q)/QC_G(Q)$ has no non-trivial
normal $p$-subgroup. Moreover the set of $p$-radical $p$-centric subgroups 
of  $P=p_{+}^{1+2}$ is closed under $G$-conjugation for any finite group $G$, 
and if this set is empty, then the $p$-fusion is controlled by the normaliser of $P$.

\begin{prop}\label{prop:prank2}
Let $p$ be an odd prime and $G$ be a quasi-simple group such that
$G/Z(G)$ is sporadic simple. If the $p$-rank of $G$ is $2$, then the
torsion-free part of $T(G)$ is as given in Table~\ref{tbl:spor2}.   
\end{prop}

\begin{table}[htbp]
 \caption{$TF(G)$ for sporadic quasi-simple groups in $p$-rank $2$, $p>2$}
  \label{tbl:spor2}
\[\begin{array}{|c|l|l||c|l|l||c|l|l|}
\hline
     G& P& TF(G)&  G& P& TF(G)& G & P & TF(G) \cr
\hline
   \spor M_{11}& 3^2 &  \IZ&                                               \spor{McL} &5_+^{1+2}  &    \textcolor{black}{\IZ}   &     	         3.\nspor{Fi}_{22}& 5^2& \IZ \cr
   \spor M_{12} & 3_+^{1+2} &  \textcolor{black}{\IZ^{3}}  &             3.\nspor{McL} &5_+^{1+2} &   \textcolor{black}{\IZ}    & 		6.\nspor{Fi}_{22}& 5^2& \IZ \cr
  2.\nspor M_{12}& 3_+^{1+2} &  \textcolor{black}{\IZ^{3}}  &
  \spor{He}& 3_+^{1+2}& \textcolor{black}{\IZ^{2}}    &
  \spor{Th} & 5_+^{1+2} & \textcolor{black}{\IZ}\cr 
  \spor M_{22} & 3^2&  \IZ  &                                                       \spor{He}& 5^2&   \IZ &							\spor{Th} & 7^2 & \IZ  \cr
 2.\nspor M_{22}& 3^2&  \IZ  &                    				   \spor{He}&  7_+^{1+2}  & \textcolor{black}{\IZ^{3}}&			\spor{Fi}_{23}& 5^2& \IZ \cr
 3.\nspor M_{22}& 3_+^{1+2}& \textcolor{black}{\IZ} &                      \spor{Ru}& 3_+^{1+2}& \textcolor{black}{\IZ}   &				\spor{Co}_{1}  & 7^2 & \IZ \cr
  4.\nspor M_{22}& 3^2&  \IZ  &                       			   2.\nspor{Ru}& 3_+^{1+2}& \textcolor{black}{\IZ}  &			2.\nspor {Co}_{1}  & 7^2 & \IZ \cr
 6.\nspor M_{22}& 3_+^{1+2}& \textcolor{black}{\IZ}&
 \spor{Ru}& 5_+^{1+2} & \textcolor{black}{\IZ^{2}}  &
 \spor J_{4} & 3_+^{1+2} &  \textcolor{black}{\IZ}     \cr 
12.\nspor M_{22}& 3_+^{1+2}& \textcolor{black}{\IZ}  &
2.\nspor{Ru}& 5_+^{1+2} & \textcolor{black}{\IZ^{2}}  &		\spor
J_{4} & 11_+^{1+2} &  \textcolor{black}{\IZ}  \cr 
 \spor J_{2} &  3_+^{1+2} &   \textcolor{black}{\IZ} &                   	   \spor{Suz}& 5^2&  \IZ &							\spor{Fi}_{24}'& 5^2& \IZ   \cr
2.\nspor J_{2} &  3_+^{1+2} &  \textcolor{black}{\IZ}    &                    2.\nspor{Suz}& 5^2&  \IZ & 						         3.\nspor{Fi}_{24}'& 5^2& \IZ \cr
 \spor J_{2} & 5^2 &  \IZ   &                    				   3.\nspor{Suz}& 5^2&  \IZ &							\spor{Fi}_{24}'& 7_+^{1+2} & \textcolor{black}{\IZ^{3}} \cr
2.\nspor J_{2} & 5^2 &  \IZ   &           					   6.\nspor{Suz}& 5^2&  \IZ &							3.\nspor{Fi}_{24}'& 7_+^{1+2} & \textcolor{black}{\IZ^{3}}\cr
  \spor M_{23}& 3^2&    \IZ  &                  				   \spor{O'N}& 7_+^{1+2}& \textcolor{black}{\IZ^{3}}   & 			\spor B & 7^2 & \IZ  \cr
  \spor{HS}& 3^2&  \IZ   &                 					   3.\nspor{O'N}& 7_+^{1+2}&\textcolor{black}{\IZ^{3}}  &		2.\nspor B & 7^2 & \IZ  \cr
2.\nspor{HS}& 3^2&  \IZ   &                     				   	  \spor{Co}_{3} & 5_+^{1+2} &   \textcolor{black}{\IZ}   &		\spor M & 11^2 & \IZ \cr
 \spor{HS}&5_+^{1+2} &   \textcolor{black}{\IZ^{2}}   &
 \spor{Co}_{2} & 5_+^{1+2} &  \textcolor{black}{\IZ}    &
 \spor M & 13_{+}^{1+2} &  \textcolor{black}{\IZ^{2}} \cr 
 2.\nspor{HS}&5_+^{1+2} &    \textcolor{black}{\IZ^{2}}  &
 \spor{Fi}_{22}& 5^2& \IZ		&
 \tw2\spor F_4(2)'&3_+^{1+2}&   \textcolor{black}{\IZ^{2}}  \cr 
 \spor M_{24} & 3_+^{1+2} &\textcolor{black}{\IZ^{2}}  &           	   2.\nspor{Fi}_{22}& 5^2& \IZ	&		                           	\tw2\spor F_4(2)'& 5^2  &  \IZ \cr
\hline
\end{array}\]
\end{table}

 \begin{proof}
By Lemma~\ref{lem:ranktf}, the rank of $TF(G)$ is equal to $n_{G}$. 
Let $P$ be a Sylow $p$-subgroup of $G$ with $p$-rank~$2$. 

If $P$ is elementary abelian, then $n_{G}=1$ and
$TF(G)=\langle\Omega\rangle\cong \IZ$ in the $28$ cases listed above. \par
Now assume $P\cong p^{1+2}_+$. So $P$ has $p+1$ elementary abelian
$p$-subgroups of rank $2$ and $n_{G}$ is the number of $G$-conjugacy
classes of such subgroups. We run through the possibilities for $p$
and $G$ according to the classification. The results mentioned below are
taken or deduced from \cite[Table 1.1, Table 1.2, Remark 1.4]{RV}. For
convenience, let us write $\CE$ for the set of the elementary abelian
$p$-subgroups of rank $2$ of $P$.

\textbf{The case $p=3$.}
In this case $|\CE|=4$. 

If $G\in\{\spor{Ru}, \spor J_{4}\}$, then all four subgroups in $\CE$
are $G$-conjugate, giving $n_{G}=1$.  

If $G=\tw2\spor F_4(2)'$, then the subgroups in $\CE$ split into two
conjugacy classes, giving $n_{G}=2$. 

If $G\in\{\spor M_{12},2.\nspor M_{12}\}$, then fusion in $G$ is the
same as in $\PSL_{3}(3)$ so that $\CE$ splits into three conjugacy
classes, giving $n_G=3$. 
More precisely, $\{\langle x,z \rangle\}$ and $\{\langle y,z \rangle\}$
are two $G$-conjugacy classes and both subgroups are $3$-centric $3$-radical.
Moreover $N_{G}(P)/PC_{G}(P)\cong2^2$ can be realised as the subgroup of 
diagonal matrices of $\GL_{2}(3)$, so that  $\langle xy,z \rangle$ and 
$\langle xy^{2},z\rangle$ become conjugate via the matrix 
$\bigl(\begin{smallmatrix}1 & 0 \\0 & 2\end{smallmatrix}\bigr)$. 

If $G\in \{3.\nspor M_{22},6.\nspor M_{22},12.\nspor
M_{22}\}$, then $3$ divides the order of the centre of $G$, and because
all the subgroups of order $3$ of a Sylow $3$-subgroup of $\spor
M_{22}$ are conjugate in $\spor M_{22}$ (cf \cite{ATLAS}), the four
subgroups of $\CE$ in each of the above covers of $\spor M_{22}$ are
$G$-conjugate. Therefore $n_G=1$. 

If $G\in\{\spor M_{24}, \spor{He}\}$, then $\CE$ splits into two
conjugacy classes of size two (cf  \cite[Proof of Lemma 4.8]{RV}), one of
which is formed by the two $3$-centric $3$-radical subgroups of $P$ in
$G$. So $n_{G}=2$. 

If $G=\spor J_{2}$, then all four subgroups in $\CE$
are $G$-conjugate, giving $n_{G}=1$. Indeed, there is no $3$-centric $3$-radical subgroup and  $N_{G}(P)/PC_{G}(P)\cong8$
can be  realised as the subgroup of $\GL_{2}(3)$ generated by the matrix
$\bigl(\begin{smallmatrix}1 & 1 \\2 & 1\end{smallmatrix}\bigr)$, which
sends $x$ to $xy^{2}$, $y$ to $xy$ and  $z$ to $z^{2}$. In consequence all four elements of $\CE$ are conjugate.

\textbf{The case $p=5$.} In this case $|\CE|=6$.

If $G=\spor{Th}$, then all six subgroups in $\CE$ are $G$-conjugate,
giving $n_{G}=1$.   

If $G\in\{\spor{Ru}, 2.\nspor{Ru}\}$, then $\CE$ splits into two
conjugacy classes of size two and four (cf \cite[Proof of
Lemma 4.8]{RV}). The conjugacy class of size $2$ is formed by the
$5$-centric $5$-radical subgroups of $P$ in $G$, and the other four are
conjugate in $N_G(P)$. Note indeed that $N_{G}(P)/PC_{G}(P)\cong
4^{2}:2$. 

In all the remaining cases \cite{RV} show that there are no $5$-centric
$5$-radical subgroups of $P$ in $G$, saying that the $p$-fusion is
controlled by $N_G(P)$. In other words the $G$- and $N_G(P)$-conjugacy
classes of the subgroups in $\CE$ coincide, and their number gives us $n_G$.
We split our analysis into four, according to the structure of
$N_{G}(P)/PC_{G}(P)$ as a subgroup of $\GL_{2}(5)$. 

If $G=\spor{Co}_{3}$, then $N_{G}(P)/PC_{G}(P)\cong D_{48}$ and all six
subgroups in $\CE$ are $G$-conjugate. Therefore $n_{G}=1$. 

If  $G=\spor{Co}_{2}$, then $N_{G}(P)/PC_{G}(P)\cong 4.\fS_{4}$ and all
six subgroups in $\CE$ are $G$-conjugate. Therefore $n_{G}=1$. 

If $G\in\{\spor{HS},2.\nspor{HS}\}$, then $N_{G}(P)/PC_{G}(P)\cong
D_{16}$ and we calculate that $\CE$ splits into two $G$-conjugacy
classes. Therefore $n_G=2$.  

If $G\in\{\spor{McL},3.\nspor{McL}\}$, then $N_{G}(P)/PC_{G}(P)\cong
3:8$ and all six subgroups in $\CE$ are $G$-conjugate. Therefore
$n_{G}=1$. 

\textbf{The case $p=7$.} In this case $|\CE|=8$. 
We split the analysis into three according to the structure of
$N_{G}(P)/PC_{G}(P)$, which is explicitly described as a subgroup of
$\GL_{2}(7)$ in \cite[Proof of Lemma 4.16]{RV} allowing us to compute
$n_{G}$ by hand. There are five groups to consider.

If  $G=\spor{He}$, then $N_{G}(P)/PC_{G}(P)\cong\fS_{3}\times 3$ and
$\CE$ splits into three conjugacy classes: two with three subgroups and
one of size two, giving $n_{G}=3$. In this case there are three
$G$-conjugate $7$-centric $7$-radical subgroups.  

If $G\in\{\spor{Fi}_{24}',3.\nspor{Fi}_{24}'\}$, then
$N_{G}(P)/PC_{G}(P)\cong\fS_{3}\times 6$ and $\CE$ splits into
three conjugacy classes: two with three subgroups and one of size two,
giving $n_{G}=3$. In this case there are six $7$-centric $7$-radical
subgroups split into two $G$-conjugacy classes of size $3$. 

If $G\in\{\spor{O'N},3.\nspor{O'N}\}$, then $N_{G}(P)/PC_{G}(P)\cong
D_{8}\times 3$ and $\CE$ splits into three conjugacy classes: two with
two subgroups and one of size four, giving $n_G=3$. In this case there
are four $7$-centric $7$-radical subgroups split into two $G$-conjugacy
classes of size $2$.

\textbf{The case $p=11$.} In this case $G=\spor J_{4}$ and
Theorem~\ref{thm:ti} says that $P$ is T.I., which implies
$N_G(P)=N_G(Z(P))$. 
In this case there are no $11$-centric $11$-radical
subgroups, and so the $11$-fusion is controlled by $N_G(P)$. 
From \cite{BENSONthesis, ATLAS}, we get that
$N_G(P)=11_{+}^{1+2}:(5\times (2.\fS_{4}))$ and $N_{G}(P)/PC_{G}(P)\cong
5\times (2.\fS_{4})$. 
Hence the twelve
elementary abelian $11$-subgroups of $P$ of rank two are $G$-conjugate,
and so $n_G=1$.

\textbf{The case $p=13$.} In this case $G=\spor M$  and
the fourteen elementary abelian $13$-subgroups of $P$ of rank two are
split into two $G$-conjugacy classes, namely one of size six, consisting
of $13$-centric $13$-radical subgroups, and another of size eight,
giving $n_{G}=2$. Here $N_{G}(P)/PC_{G}(P)\cong 3\times (4.\fS_{4})$
(cf \cite[Proof of Lemma 4.18]{RV}). 
\end{proof}

\subsection{Groups with $p$-rank greater or equal to
  $3$}\label{sec:prankgt2} 

\begin{prop}\label{lem:rktf=1}
Let $G$ be a quasi-simple group such that $G/Z(G)$ is sporadic simple
and let  $p=char(k)=3$. If   $G$ is one of the groups $3.\nspor{J}_{3}$,
$\spor{McL}$, $3.\nspor{McL}$, $\spor{Suz}$, $2.\nspor{Suz}$,
$3.\nspor{Suz}$, $\spor{O'N}$,  $\spor{Co}_{3}$,  $\spor{Co}_{2}$,
$\spor{Fi}_{22}$, $2.\nspor{Fi}_{22}$, $3.\nspor{Fi}_{22}$, $\spor{HN}$,
$\spor{Ly}$, $\spor{Th}$, $\spor{Fi}_{23}$, $\spor{Co}_{1}$,
$2.\nspor{Co}_{1}$, $\spor{Fi}'_{24}$, $3.\nspor{Fi}'_{24}$, $\spor{B}$,
$2.\nspor{B}$, $\spor{M}$,  then $TF(G)=\langle\Omega\rangle\cong \IZ$. 
\end{prop}

\begin{proof}
In all cases the $3$-rank of $G$ is greater or equal to $4$, see
\cite[Table 5.6.1]{GLS3}. Therefore $n_{G}=1$ by
Theorem~\ref{thm:connected} and the claim follows from
Lemma~\ref{lem:ranktf}. 
\end{proof}

The remaining cases are the quasi-simple groups $G$ such that $G/Z(G)$
is sporadic simple and $G$ has $p$-rank between $3$ and $p$.

For convenience, we introduce a complementary notation taken from
\cite{ATLAS}. Namely, we write $nX^e$, where $n,e$ are positive integers and $X$ a
capital letter, to denote an elementary abelian group $n^e$ spanned by
elements belonging to the $G$-conjugacy class labelled $nX$ in \cite{ATLAS}.

\begin{prop}
Let $p\geq 3$ and $G$ be a quasi-simple group such that $G/Z(G)$ is 
sporadic simple.  If the $p$-rank of $G$ is greater or equal to $3$ and smaller 
or equal to $p$, then $G$ and $TF(G)$ are as listed in Table~\ref{tab:spor2}.
\end{prop}

\begin{table}[htbp]
\caption{Sporadic quasi-simple groups with $p$-rank between $3$ and
  $p$.}\label{tab:spor2}
\[\begin{array}{|c|c|c|c||c|c|c|c||c|c|c|c|}
\hline
     G& p& p\text{-rank} & TF(G)&  G& p& p\text{-rank}& TF(G)& G & p & p\text{-rank}& TF(G) \cr
\hline
   \spor{J}_{3} & 3& 3& \IZ &			\spor{Ly} &5 & 3& \IZ   &  		         \spor{M} & 5& 4& \IZ \cr
   3.\nspor{O'N} & 3& 3& \IZ  &            		\spor{HN} &5 &3&  \IZ   &     	 \spor{M}        & 7& 3&  \IZ \cr
    \spor{Co}_{1} & 5& 3&  \IZ^{2}&       		\spor{B} &5 & 3&  \IZ &     	                  & & &  \cr  
   2.\nspor{Co}_{1} & 5& 3&  \IZ^{2} & 		2.\nspor{B} &5 & 3& \IZ &     	         & & & \cr  
\hline
\end{array}\]
\end{table}

\begin{proof}
A Sylow $3$-subgroup of $\spor J_3$ has a non-cyclic center and
therefore $\spor J_3$ cannot have a maximal elementary abelian
$3$-subgroup of rank $2$. 
For $3.\nspor{O'N}$, $\spor{Co}_{1}$ and
$2.\nspor{Co}_{1}$, we used MAGMA \cite{MAGMA} to calculate the number 
of conjugacy
classes of the elementary abelian $p$-subgroups of rank $2$, for the
appropriate prime $p$, and then invoked Lemma~\ref{lem:ranktf}.
For the remaining cases, we gather information from \cite{ATLAS,GLS3}.
Let $P$ be a Sylow $p$-subgroup of $G$. If $G=\spor{Ly}$ and $p=5$ then 
$P$ is contained in the normaliser 
$N_G(5A^3)=5^3.\PSL_3(5)$ of an elementary abelian subgroup $5A^3$. So
$\PSL_3(5)$ acts as automorphism group of a $3$-dimensional vector space
$\F_5^3$ and by linear algebra, using Jordan forms, any $5$-element of
$\PSL_3(5)$ must centralise a non-trivial subgroup in $5^3$ (i.e. must have
a non-trivial eigenspace on $\F_5^3$). 
Therefore each element of order $5$ in $P$ has a centraliser in
$P$ of rank at least $3$, saying that $P$ has no maximal elementary
abelian subgroup of rank $2$. A similar argument applies to $\spor B$
and $2.\nspor B$, as $\spor B$ has a local subgroup of the form
$N_G(5B^3)=5^3.\PSL_3(5)$. 
If $G=\spor{HN}$ and $p=5$ we look at the normalisers of
the elements of order $5$ as given in \cite{GLS3}. In the Sylow
$5$-subgroup of $N_G(5A)=(D_{10}\times \PSU_3(5)):2$ each element of
order $5$ is centralised by a subgroup of $5$-rank at least $3$. From
$N_G(5B)$ we see that $C_P(5B)$ has rank at least $3$ and similarly for
$\overline{5C}$ and $5E$. Consequently $G$ cannot have a maximal
elementary abelian 5-subgroup of rank $2$.
Finally, if  $G=\spor M$ and $p=5$ then
$P$ is a Sylow $5$-subgroup of a local subgroup of the form
$5^{1+6}_+.(2.\spor{J}_2).4$ which has no maximal elementary abelian
subgroups of rank $2$ and so neither does $\spor M$. If $p=7$ then there
are two conjugacy classes of elements of order $7$. The normalisers have
the form $((7.3)\times\spor{He}).2$ and $7^{1+4}_+.(2.\fA_7\times3).2$ and
none of these has a maximal elementary abelian $7$-subgroup of rank $2$.
\end{proof}


\section{The torsion subgroup $TT(G)$ of $T(G)$ in odd
  characteristic.}\label{sec:TT}  

Let $p$ be an odd prime. We consider in this section groups $G$ with a noncyclic Sylow
$p$-subgroup $P$. (If $P$ is cyclic, then we refer to Section~\ref{sec:tables}.) 
By Lemma~\ref{lem:omnibus}, the torsion
subgroup $TT(G)$ identifies with a subgroup of $X(N_{G}(P))$ via
restriction and consists precisely of the $kG$-Green correspondents of
the $1$-dimensional $kN_{G}(P)$-modules that are endotrivial, i.e. 
the trivial source endotrivial modules. Recall that for a group $G$
we write $X(G)$ for the group of isomorphism classes of $1$-dimensional
$kG$-modules, and it is isomorphic to the $p'$-group $G/G'P$ where
$G'=[G,G]$ is the commutator subgroup of $G$.

Simple endotrivial modules for covering groups of sporadic groups have
been determined in \cite[\S 7]{LMS}, where the authors prove that many
of them are trivial source modules, hence elements of $T(G)$. 
We now continue our analysis of $TT(G)$ along similar lines, using Green
correspondence and induction and restriction of characters.

In the rest of this section we use the following notation. If $G$ is a
given sporadic group and $\ell.G$ an $\ell$-fold cover of $G$, we write
$P$ and $\wt P$ for a Sylow $p$-subgroup of $G$ and of $\ell.G$
respectively, and set  $N=N_{G}(P)$ and $\wt N=N_{\ell.G}(\wt P)$.
Also, for a subgroup $H$ of $G$ containing $N$ and for
an indecomposable $kN$-module $V$, then $\Gamma_{H}(V)$ denotes the
$kH$-Green correspondent of $V$. 
Similarly for $\ell.G$.

We use the structure of the Sylow normalisers and maximal subgroups as
given in \cite{ATLAS,GL,GLS3} and the GAP Character Table Libraries
\cite{CTblLib}. These are especially useful in determining $X(N)$.
In our computations we denote by $\chi$ ordinary irreducible characters of a group. 
We index the ordinary irreducible characters and the blocks according 
to the GAP Character Table Libraries \cite{CTblLib}, with the exception that 
we denote the principal block by $B_{0}$ instead of $B_{1}$. Accordingly we denote by
 $e_{i}$ the central primitive idempotent corresponding to the block $B_{i}$.
We denote by $1_{H}\in\Irr(H)$ the trivial character of a group $H$ and 
$1_{i}\in X(H)$ denotes a non-trivial $1$-dimensional $kH$-module. 
We also identify these with the corresponding ordinary and Brauer
characters, as well as the corresponding simple $kG$-module and its
class in $X(H)$.

Throughout we use the dimensional criterion for trivial source
endotrivial modules of Lemma~\ref{basics}(7),
the possible character values for characters of endotrivial modules 
given by Theorem~\ref{thm:lift} and Theorem~\ref{lem:Greenchars}, and
the possible values of characters of trivial source modules given by
Lemma~\ref{lem:LandScott}.  We split the cases according to the values
of $p$ and the arguments we used to obtain $TT(G)$. \\

First, as for the $2$-fold covers in characteristic $2$, the $3$-fold
covers in characteristic $3$ cannot have non-trivial torsion endotrivial
modules.  
\begin{lemma}\label{lem:3foldcov}
Let $p=3$ and let $G$ be one of the covering groups $3.\nspor{M}_{22}$,
$6.\nspor{M}_{22}$, $12.\nspor{M}_{22}$, $3.\nspor{J}_{3}$,
$3.\nspor{McL}$, $3.\nspor{Suz}$, $6.\nspor{Suz}$, $3.\nspor{O'N}$,
$3.\nspor{Fi}_{22}$, $6.\spor{Fi}_{22}$, $3.\nspor{Fi}'_{24}$. Then
$TT(G)=\{[k]\}$. 
\end{lemma}

\begin{proof}
In all cases $G$ has a non-trivial normal $3$-subgroup so that
$TT(G)=\{[k]\}$ by Lemma~\ref{lem:cext}(2).  
\end{proof}

Next we treat covering groups $G$ such that $p\,\nmid\,|Z(G)|$.

\begin{lemma}\label{lem:nofaithful}
Assume  $p=3$ and  $G$ is one of the covering groups $2.\nspor{M}_{12}$,
$4.\nspor{M}_{22}$, $2.\nspor{J}_{2}$, $2.\spor{HS}$, $2.\nspor{Suz}$,
$2.\nspor{Fi}_{22}$, $2.\nspor{Co}_{1}$, $2.\nspor{B}$, or $p=5$ and
$G$ is one of the covering groups $2.\nspor{J}_{2}$, $2.\nspor{HS}$,
$2.\nspor{Ru}$, $2.\nspor{Suz}$, $3.\nspor{Suz}$, $6.\nspor{Suz}$,
$2.\nspor{Co}_{1}$, $3.\nspor{Fi}'_{24}$, $2.\nspor{B}$, or $p=7$ and
$G$ is one of the covering groups $2.\nspor{Co}_{1}$,
$3.\nspor{Fi}'_{24}$, $2.\nspor{B}$.  
Then $G$ has no faithful endotrivial module. 
Furthermore in all cases $T(G)\cong T(G/Z(G))$ via inflation, except if
$G=4.\nspor{M}_{22}$ in which case $T(G)\cong T(2.\nspor{M_{22}})$. 
\end{lemma}

\begin{proof}
Since $p\,\nmid\,|Z(G)|$ faithful endotrivial modules must afford
characters in faithful blocks. The proof is similar to that of
Lemma~\ref{lem:coverschar2}. Except for $2.\nspor{HS}$ in characteristic
$3$,  \cite{CTblLib} shows in all cases
that $G$ has a $p$-singular conjugacy class $C$ such that for all
faithful characters $\psi\in\Irr(G)$ we have $\psi(c)\in m\IZ$ with
$m\in\IZ\setminus\{\pm1\}$, $c\in C$. Hence we obtain a contradiction with the
possible values of characters of endotrivial modules given by
Theorem~\ref{thm:lift}.
More accurately, we can choose $m$ and $C$ as follows. For $p=3$ take
$m=0$ if $G$ is one of $2.\nspor{M}_{12}$,  $4.\nspor{M}_{22}$,
$2.\nspor{J}_{2}$, $2.\nspor{Suz}$, $2.\nspor{Fi}_{22}$,
$2.\nspor{Co}_{1}$, or $2.\nspor{B}$, and for $C$ choose $12a$, $6b$, $12a$,
$12a$, $12a$, $12a$, $6c$ respectively. For $p=3$ and $G=2.\nspor{M}_{12}$ take
$(m,C)=(2,6c)$. If $p=5$ take $m=0$ if $G$ is one of $2.\nspor{J}_{2}$,
$2.\nspor{HS}$, $2.\nspor{Suz}$, $3.\nspor{Suz}$, $6.\nspor{Suz}$,
$2.\nspor{Co}_{1}$, or $2.\nspor{B}$ and for $C$ choose $20a$, $20a$, $20a$,
$15e$, $15e$, $20a$, $20b$ respectively. For $p=5$, if $G=2.\nspor{Ru}$ take
$(m,C)=(5,5b)$, and if $G=3.\nspor{Fi}'_{24}$ take
$(m,C)=(5,15c)$. If $p=7$ take $(m,C)=(2,7a)$ for
$G=2.\nspor{Co}_{1}$, take $(m,C)=(3,7a)$ for
$G=3.\nspor{Fi}'_{24}$, and take $(m,C)=(0,14b)$ for
$G=2.\nspor{B}$. 
Finally if $G=2.\nspor{HS}$ and $p=3$, then all faithful characters
take value in $3i\IZ$ on class $12a$, hence the result. (Notation for
conjugacy classes is that of GAP.)  
The claim on $T(G)$ is straightforward from Lemma~\ref{lem:inf}.
\end{proof}

We now go through the list of sporadic simple groups by increasing order, 
including non-trivial covers not treated by Lemmas~\ref{lem:3foldcov} and~\ref{lem:nofaithful} above, 
and compute the torsion subgroup $TT(G)$ for as many groups $G$ and characteristics $p$ as possible.
In the cases when we do not obtain the full structure of $TT(G)$ we give
bounds instead.

Let us point out that the results for the groups $\spor{He}$ and $\spor{Suz}$ 
in characteristic~$5$,  hereafter mentioned as `MAGMA computation', partly use a 
method recently developed by Paul Balmer \cite{Balmer} and
a subsequent algorithm by Carlson (private communication). Balmer's method
determines the kernel of the restriction $TT(G)\to TT(H)$ to a subgroup
$H$ of $G$ by showing that this kernel is isomorphic to the set of the
so-called weak $H$-homomorphisms. In particular, for $H=P$ the set of weak
$P$-homomorphisms is isomorphic to the subgroup of $TT(G)$ spanned by
the trivial source endotrivial modules. We omit here the technicalities, to
avoid further deviation from our present scope, referring the reader to
\cite{Balmer}. The main drawback of this method is that it cannot always
be used in practice, and when it does, then it often
relies on tedious computer calculations. Both authors are thankful to
Jon Carlson with his help in implementing an algorithm in MAGMA and
running it on specific groups.
\label{page:balmer}

\subsection{The group $G=\spor{M_{11}}$ in characteristic
  $3$}\label{sec:m11tors} 

By Theorem~\ref{thm:ti}, the group $G$ has an abelian T.I. Sylow
$3$-subgroup $P\cong 3^2$ and $N\cong 3^2 : SD_{16}$.
Thus $X(N)\cong SD_{16}/(SD_{16})'\cong (\IZ/2)^2$.
It follows from Lemma~\ref{lem:omnibus}(3),(4) and
Proposition~\ref{prop:prank2} that $TT(G)\cong X(N)\cong (\IZ/2)^{2}$.

 The indecomposable modules representing the classes in $TT(G)$ are 
 the $kG$-Green correspondents of the modules in $X(N)$. Their Loewy and 
 socle series, as well as the characters they afford, are described
 in~\cite[(4.3)]{KWgreencorresp}. Their dimensions are $1$, $55$, $55$
 and  $10$.

\subsection{The group $G=\spor{M_{12}}$ and its double cover in
  characteristic $3$}  \label{sec:m12tors}  

In this case $P=3^{1+2}_+$ and we have a chain of subgroups 
$$P\leq N\leq N_2\leq \spor M_{12}
\qbox{where $N\cong 3_+^{1+2}:2^2$ and $N_2\cong 3^2:\GL_2(3)$.}$$  
So $X(N)\cong (\IZ/2)^2$ and $X(N_2)\cong \IZ/2$.
Since $|N_2:N|=4$, induction from $kN$ to $kN_2$ of a $1$-dimensional
$kN$-module is a module of dimension $4$. Hence the $kN_{2}$-Green
correspondents
$\Gamma_{N_{2}}(1_{N}),\Gamma_{N_2}(1_1),\Gamma_{N_2}(1_2),\Gamma_{N_2}(1_3)$
of the four $1$-dimen\-sional modules $1_{N},1_1,1_2,1_3\in X(N)$ have
dimension $1,1,4$ and $4$, respectively, because $N_2$ has exactly two
non-isomorphic $1$-dimensional modules. In particular we conclude that
$\Gamma_{N_2}(1_2)$, $\Gamma_{N_2}(1_{3})$ are not endotrivial, for $4\not\equiv
1\pmod{|P|}$. A fortiori $\Gamma_G(1_2)$, $\Gamma_G(1_3)$ are not endotrivial either. It
remains to prove that $\Gamma_G(1_1)\notin TT(G)$. Note that the principal block $B_{0}$ is the unique block of $kG$ with full defect. The Loewy and
socle series of $\Gamma_G(1_1)$ have been computed in \cite[6.1]{KWloewy}.
In particular, if $\psi_{1}$ denotes the linear character corresponding
to $\Gamma_{N_2}(1_1)$, then we see that
$e_{0}\cdot\Ind_{N_{2}}^{G}(\psi_{1})$ yields that $\Gamma_{G}(1_1)$ is
the trivial source module affording the character $\chi_9+\chi_{13}$ of degree
$175\not\equiv 1\pmod{|P|}$. Whence
$\Gamma_{G}(1_1)\notin TT(G)$ and $TT(G)=\{[k]\}$.

\subsection{The group $G=\spor{M_{22}}$ and its covers in characteristic
  $3$}\label{sec:m22tors} 

A Sylow $3$-subgroup $P$ of $G$ is elementary abelian of order $9$
and $N\cong 3^{2}:Q_{8}$. So 
$X(N)\cong (\IZ/2)^2$. In order to obtain $TT(G)$, we first  consider the 
group $2.G=2.\spor{M}_{22}$. A Sylow $3$-subgroup $\wt P$ of $2.G$ is elementary 
abelian of order $9$ and $X(\wt N)\cong \IZ/2\oplus \IZ/4$. We gather from \cite[Thm. 7.1]{LMS} and \cite[Sec. 2]{LM} that
$2.G$ has six simple trivial source endotrivial modules. Namely,  the trivial module $k$, a self-dual module $S_{55}$ of dimension 55 affording the character $\chi_{5}\in\Irr(2.G)$, and four faithful modules:  $S_{10}$, $S_{10}'$, both of dimension $10$, affording  the
characters $\chi_{13}, \chi_{14} \in \Irr(2.G)$, and $S_{154}$, $(S_{154})^*$ both of dimension $154$, affording the characters $\chi_{19}, \chi_{20} \in \Irr(2.G)$.
It follows directly that $TT(2.G)\cong X(\wt N)\cong\IZ/2\oplus\IZ/4$. Furthermore, by Lemma~\ref{lem:inf}(2), Green
correspondence preserves the faithfulness of blocks, hence $TT(G)\cong (\IZ/2)^{2}$. The group $T(4.\nspor{M}_{22})$ is isomorphic to $T(2.\nspor{M}_{22})$
via inflation by Lemma~\ref{lem:nofaithful}. However, notice that
$TT(4.\nspor{M}_{22})\ncong X(\wt N)\cong \IZ/2\oplus\IZ/8$ in this
case. Finally the groups $TT(3.\nspor{M}_{22})$, $TT(6.\nspor{M}_{22})$
and $TT(12.\nspor{M}_{22})$ are all trivial by Lemma~\ref{lem:cext}(2).


\subsection{The group $G=\spor{J_{2}}$}\label{sec:j2}\

\textbf{Characteristic $3$}.
In this case $P\cong 3_{+}^{1+2}$, $N\cong 3_{+}^{1+2}:8$, $X(N)\cong \IZ/8$ and $B_{0}(kG)$ is the unique $3$-block with full defect.
We pick $1_{4}$ to be a character of order $4$ of $X(N)$, then
$e_{0}\cdot\Ind_{N}^{G}(1_{4})=\chi_{12}+\chi_{13}+2\chi_{21}$. By the
criteria on degrees and character values $\Gamma_{G}(1_{4})$ cannot  be
endotrivial. It follows that $TT(G)$ can at most be cyclic of order
$2$. Let $1_{7}$ denote the linear character of order $2$ of $N$. We
have $e_{0}\cdot\Ind_{N}^{G}(1_{7})=\chi_{4}+\chi_{5}+\chi_{12}+2\chi_{13}+\chi_{16}+\chi_{17}+\chi_{20}+2\chi_{21}$. Again
using the possible degrees and values of trivial source characters given
by Lemma~\ref{lem:LandScott} it is easy to see that $\Gamma_{G}(1_{7})$
must afford the character $\chi_{4}+\chi_{5}+\chi_{13}$, which takes
value $1$ on all the non-trivial $3$-elements and is thus endotrivial
Theorem~\ref{lem:Greenchars}. We conclude that $TT(G)\cong
\IZ/2$. 

\textbf{Characteristic $5$}.
In this case $P$ is elementary abelian of rank $2$ and $N\cong
5^{2}:D_{12}$ so that $X(N)\cong (\IZ/2)^{2}$.  The principal $5$-block of
$kG$ is the unique block with full defect. Let  $X(N)=\{1_{N}, 1_{1},
1_{3},1_{4}\}$. Inducing the non-trivial characters in $X(N)$ we get the
following. 
First $e_{0}\cdot\Ind_{N}^{G}(1_{1})=\chi_{14}+\chi_{15}+\chi_{19}$, so
that  by the degree and character value criteria we obtain that
$\Gamma_{G}(1_{1})$ affords $e_{0}\cdot\Ind_{N}^{G}(1_{1})$, which does
not take value $1$ on all non-trivial $5$-elements. 
Second $e_{0}\cdot \Ind_{N}^{G}(1_{3})=\chi_{14}+\chi_{15}+\chi_{16}+\chi_{17}$. Again by the degree
criteria, we must have that $\Gamma_{G}(1_{3})$ affords
$e_{0}\cdot \Ind_{N}^{G}(1_{3})$, as it takes value $1$ on all the non-trivial
$5$-elements. By Theorem~\ref{lem:Greenchars}, $\Gamma_{G}(1_{3})$
is endotrivial whereas $\Gamma_{G}(1_{1})$ is not. Whence $TT(G)\cong
\IZ/2$.  


\subsection{The group $G=\spor{M_{23}}$ in characteristic
  $3$}\label{sec:m23tors} 

In this case $P$ is elementary abelian of order $9$ and 
$N\cong 3^{2}:Q_{8}.2$. It follows that $X(N)\cong (\IZ/2)^2$.

By \cite[Thm. 7.1]{LMS},  $\spor{M}_{23}$ has a simple self-dual
endotrivial module $S_{253}$ of dimension $253$.
Hence $\IZ/2\cong \left\langle [S_{253}]\right\rangle \leq TT(G)$.

Now, the natural permutation $kG$-module on $23$ points has a
$22$-dimensional composition factor $S_{22}$, which must be a direct
summand. Therefore $S_{22}$ is a trivial source module with vertex $P$
and its $kN$-Green correspondent is simple (see \cite[\S
3.7]{DANZKUEL}). 
Thus it must have dimension at most $2$. Since $22\equiv 1\pmod{3}$, we
conclude that $S_{22}$ is the $kG$-Green correspondent of a
$1$-dimensional $kN$-module. But as $22\not\equiv 1\pmod{|P|}$, the
module $S_{22}$ is not endotrivial.  Consequently $TT(G)\cong \IZ/2$.

\subsection{The group $G=\tw2\,{\spor F}_4(2)'$}\label{sec:2f42'}\

\textbf{Characteristic $3$}. In this case $P$ is extraspecial of order
$27$ and exponent $3$ and $N\cong 3^{1+2}_+ : D_8$. 
Hence $X(N)\cong (\IZ/2)^2$ and we denote $1_{N}, 1_{1},1_{2},1_{3}$  the four
non-isomorphic $1$-dimensional $kN$-modules. 

The group $G$ has a maximal subgroup $H\cong \PSL_{3}(3).2$ of index $1600$
which contains $N$ and $X(H)=\{1_{H},1_{a}\}$ has order $2$. 
We may assume that $\Gamma_{H}(1_{N})=1_{H}$ and
$\Gamma_{H}(1_{1})=1_{a}$. 
Induction from $N$ to $H$ yields
$$\Ind_{N}^{H}(1_{2})=\chi_{5}+\chi_{13}\qbox{and}
\Ind_{N}^{H}(1_{3})=\chi_{6}+\chi_{14}\,.$$ 
Since
$\chi_{5}(1)=\chi_{6}(1)=13$ and $\chi_{13}(1)=\chi_{14}(1)=39$,  and 
$13,39, 13+39\not\equiv1\pmod{|P|}$,
the $kH$-Green correspondents $\Gamma_H(1_{2})$ and $\Gamma_H(1_{3})$ 
cannot be endotrivial.
Similarly, $\Ind_{H}^{G}(1_{a})=\chi_{7}+\chi_{19}$ where
$\chi_{7}(1)=300\not\equiv 1\pmod{|P|}$, $\chi_{19}(1)=1300\not\equiv
1\pmod{|P|}$ and $(\chi_{7}+\chi_{19})(1)\not\equiv1\pmod{|P|}$ so that
the $kG$-Green correspondent of $1_{1}$ is not endotrivial. 
Therefore $TT(H)\cong \IZ/2$ and $TT(G)\cong \{[k]\}$.

\textbf{Characteristic 5}. In this case $P$ is elementary abelian of
rank $2$ and ${N\cong(5^2):(4.\fA_{4})}$. 
Theorem \ref{thm:ti} says that $P$ is a T.I. set, so that by
Lemma~\ref{lem:omnibus}(4),(5) we get that
$TT(G)\cong TT(N)\cong X(N)$ because $N$ is strongly
$p$-embedded in $G$. Finally we compute that $X(N)\cong \IZ/6$.

\subsection{The group $G=\spor{HS}$}\label{sec:hstors3}\

\textbf{Characteristic 3}. 
In this case $P\cong 3^{2}$ and ${N\cong
2\times(3^{2}.SD_{16})}$. So $X(N)\cong (\IZ/2)^{3}$. 
By \cite[Thm. 7.1]{LMS}, $G$ has three simple endotrivial modules
$S_{4},S_{5},S_{6}$ all of dimension $154$, and affording the characters
$\chi_{4},\chi_{5},\chi_{6}$ respectively. From \cite[Sec.2]{LM} we have
$TT(G)= \langle[S_{4}], [S_{5}],[S_{6}]\rangle\cong (\IZ/2)^{2}$.  
Alternatively, a MAGMA computation (not involving Balmer's method described above) 
shows that there exists a $kN$-module $1_{a}\in X(N)$ such that its $kG$-Green 
correspondent is of dimension~$175\not\equiv 1\pmod{9}$, 
so that $TT(G)$ certainly cannot have order $8$.

\textbf{Characteristic 5}. 
In this case $P\cong 5_{+}^{1+2}$, and $N\cong 5_{+}^{1+2}:(8.2)$. Also $G$ has a
maximal subgroup $H^*\cong \PSU_{3}(5).2$ containing $N$.  We get  $X(N)\cong
\IZ/2\oplus \IZ/4$, $X(H^*)\cong \IZ/2$. We note that $N$ is strongly $5$-embedded in $H^*$, so that $TT(H^*)\cong \IZ/2\oplus\IZ/4$ by Lemma~\ref{lem:omnibus}.  
Now if $1_{a}\in X(H^*)$ is non-trivial, we have $e_{0}\cdot \Ind_{H^*}^{G}(1_{a})=\chi_{2}+\chi_{5}$, where $\chi_{2}(1)=22,\chi_{5}(1)=154$. Therefore 
$\Gamma_{G}(1_{1})=e_{0} \cdot \Ind_{H^*}^{G}(1_{a})$ and affords $\chi_{2}+\chi_{5}$, but is not endotrivial since its dimension is $176\not\equiv 1\pmod{|P|}$.
From the group structure of $X(N)$, we see that $TT(G)\leq \IZ/4$.  Now the linear character $1_{5}\in X(N)$ has order $4$ and 
$$e_{0}\cdot \Ind_{N}^{G}(1_{5})=\chi_{8}+\chi_{9}+\chi_{10}+\chi_{16}+\chi_{17}+2\cdot\chi_{22}\,.$$
(Note that $B_{0}$ is the unique $5$-block with full defect.) It follows
from the possible values of trivial source modules given by
Lemma~\ref{lem:LandScott} and the $5$-decomposition matrix of $G$ that
$\Gamma_{G}(1_{5})$ affords either the character $\chi_{8}+\chi_{10}$
or the character $\chi_{8}+\chi_{22}$. In both cases we obtain that
$\Gamma_{G}(1_{5})$ is endotrivial by Theorem~\ref{lem:Greenchars}
and we conclude that $TT(G)\cong \IZ/4$.


\subsection{The group $G=\spor{J_{3}}$ in characteristic 3}\label{sec:j3}\
In this case $P$ has order $3^{5}$ and $N\cong 3^{2}.(3^{1+2}):8$ so
that $X(N)\cong \IZ/8$ (note that $N$ is maximal in $G$).  The linear
character $1_{4}$ has order $4$ in $X(N)$ and 
$$e_{0}\cdot \Indhg{N}{G}(1_{4})=2\chi_{10}+2\chi_{13}+\chi_{14}+\chi_{15}+\chi_{16}+\chi_{17}+\chi_{18}+\chi_{19}\,.$$
It follows from the character table of $J_{3}$ that if
$\dim_{k}\Gamma_{G}(1_{4})\equiv 1\pmod{|P|}$, then its character cannot
take value 1 on all non-trivial $3$-elements. Hence, by
Theorem~\ref{lem:Greenchars}, $\Gamma_{G}(1_{4})$ is not
endotrivial. This proves that $TT(G)\leq \IZ/2$. 

\subsection{The group $G=\spor{M_{24}}$ in characteristic
  $3$}\label{sec:m12tors} We show that $TT(G)=\{[k]\}$.

In this case $P$ is extraspecial of order $27$ and exponent $3$ and
$N\cong 3_{+}^{1+2}:D_{8}$ where $D_8$ is a dihedral group of
order $8$.
From \cite{ATLAS}, we see that $G$ has a maximal subgroup $H\cong
M_{12}:2$ containing $N$. Therefore $X(H)\cong \IZ/2$ and $X(N)\cong (\IZ/2)^2$. 
We denote by $1_{H}, 1_{a}$ the elements of $X(H)$
and by $1_{N}, 1_{1}, 1_{2}$ and $1_{3}$ the elements of $X(N)$. 
By Green correspondence, we may assume that
$\Gamma_{H}(1_{N})=1_{H}$ and  $\Gamma_{H}(1_{1})=1_{a}$.
Inducing the character $1_{3}$ to $H$ gives
$$e_{0}\cdot \Ind_{N}^{H}(1_{3})=\chi_{12}+\chi_{15}+\chi_{16}+\chi_{17}+\chi_{21}\,.$$ 
(Note that the principal block is the unique $3$-block of $kH$ with full defect).
Using
Theorem~\ref{lem:Greenchars} and Lemma~\ref{lem:LandScott} it is easy  to check that  for degree
reasons, or character values on $3$-elements, $\Gamma_{H}(1_{3})$ cannot
be endotrivial. Consequently $TT(H)=X(H)\cong\IZ/2$.

Finally $\Ind_{H}^{G}(1_{a})=\chi_{2}+\chi_{17}$, where $\chi_{2}(1)=23$ and
$\chi_{17}(1)=1265\equiv 23\pmod{|P|}$, so that again for degree reasons
the $kG$-Green correspondent of $1_{a}$ is not endotrivial either. The
claim follows.

\subsection{The group $G=\spor{McL}$ and its 3-fold cover.}\label{sec:mcltors5}\
 
\textbf{Characteristic 3}. In this case $P$ has order $3^{6}$ and
$N\cong 3^{4}:(3^{2}:Q_{8})$ so that $X(N)\cong (\IZ/2)^{2}$. In addition $G$
has a subgroup $H\cong 3^{1+4}:(2.\fS_{5})$. Since $|H:N|=10<|P|$, we have
$TT(H)\cong X(H)\cong\IZ/2$. So let $1_{a}$ be the non-trivial linear
character of $H$. Then
$\Indhg{H}{G}(1_{a})=\chi_{3}+\chi_{15}+\chi_{20}$. Now as
$\chi_{3},\chi_{15},\chi_{20}$ all take a value strictly greater than~$1$
on conjugacy class $3a$, $\Gamma_{G}(1_{a})$ cannot be endotrivial by
Theorem~\ref{lem:Greenchars}. Whence $TT(G)=\{[k]\}$, whereas the group
$TT(3.\nspor{McL})$ is trivial by Lemma~\ref{lem:cext}(2).  
 
\textbf{Characteristic 5}. 
In this case $P=5_{+}^{1+2}$ and $N\cong5_{+}^{1+2}:3:8$ is a maximal
subgroup of $G$. Moreover $P$ is a T.I. set by Theorem~\ref{thm:ti} and
therefore by Lemma~\ref{lem:omnibus}(4),(5) we get that $TT(G)\cong  X(N)\cong \IZ/8$.
The situation is similar for the $3$-fold cover $3.\nspor{McL}$,
i.e. $\wt P\cong 5^{1+2}_+$ and is a T.I. set as well. The normaliser
$\wt N$ is an extension of degree $3$ of $N$ and $X(\wt N)\cong
\IZ/24$. Whence $TT(3.\nspor{McL})\cong  X(\wt N)\cong \IZ/24$.

\subsection{The group $G=\spor{He}$}\label{sec:hetors5}\

\textbf{Characteristic 3}.  In this case $P\cong 3_{+}^{1+2}$ 
and  $N\cong 3_{+}^{1+2}:D_{8}$ so that $X(N)\cong (\IZ/2)^{2}$. 
The group $G$ has a maximal subgroup $H\cong 3.A_{7}.2$ 
containing $N$. It follows from Lemma~\ref{lem:omnibus}(3) that  
$TT(H)\cong X(H)$, where $X(H)=\{1_{H},1_{2}\}\cong\IZ/2$. Hence 
$TT(G)\leq TT(H)\cong \IZ/2$ via restriction and we only need 
to check whether $\Gamma_{G}(1_{2})$ is endotrivial. But  
$\Ind_{H}^{G}(1_{2})=\chi_{14}+\chi_{15}+\chi_{22}+2\cdot \chi_{29}+\chi_{30}+\chi_{31}$
so that it follows easily from the values of 
$\chi_{14},\chi_{15},\chi_{22},\chi_{29},\chi_{30},\chi_{31}\in\Irr(G)$ 
that the character afforded by $\Gamma_{G}(1_{2})$ cannot take 
value $1$ on conjugacy class $3b$. Whence $TT(G)\cong \{[k]\}$.

\textbf{Characteristic 5}.  
In this case $P$ is elementary abelian of order $25$ and
$N\cong5^{2}:(4.\fS_{4})$. So $X(N)\cong \IZ/6$.
By \cite[Thm. 7.1]{LMS}, $G$ has two simple trivial source endotrivial
modules $S_{51}$ and $S_{51}^{*}$, both of dimension $51$, affording the
characters $\chi_{2},\chi_{3}\in\Irr(G)$. Now $\chi_{1}\res{}{N}$ and
$\chi_{2}\res{}{N}$ have exactly two linear constituents $\chi_{4}$ and
$\chi_{6}$, both of which have order $3$ in $X(N)$. Therefore
$S_{51}$, $S_{51}^{*}$ have order $3$ in $TT(G)$, since they are the
$kG$-Green correspondents of $1$-dimensional modules of order $3$ in
$X(N)$, showing that $TT(G)$ has order at least $3$.
Finally a MAGMA computation shows that $TT(G)\cong \IZ/3$.

\subsection{The group $G=\spor{Ru}$ and its cover in
  characteristic $3$}\label{sec:hetors5}

In this case $P$ is extraspecial of order $27$ and exponent $3$ and
$N\cong 3_{+}^{1+2}:SD(16)$. So $X(N)\cong (\IZ/2)^2$. For $2.G$, we have $\wt
P\cong P$ and $X(\wt N)\cong\IZ/2\oplus\IZ/4$.

By \cite[Thm. 7.1]{LMS}, $G$ has a self-dual simple trivial source
endotrivial module $S_{406}$ of dimension $406$, affording the character
$\chi_{4}\in\Irr(G)$. Therefore $\IZ/2\cong \langle[S_{406}]\rangle\leq
TT(G)$.
Again by \cite[Thm. 7.1]{LMS}, $2.G$ has two faithful simple endotrivial
modules $S_{28}$ and $S_{28}^{*}$, both of dimension $28$, affording the
characters $\chi_{37},\chi_{38}\in\Irr(2.G)$. Moreover
$\chi_{37}\otimes\chi_{37}=\chi_{2}+\chi_{4}$ where $\chi_{2}$ has
defect zero. Reduction modulo $3$ yields
$S_{28}\otimes S_{28}\cong S_{406}\oplus\proj$ saying that $S_{28}$ is
a trivial source module of order $4$ in $TT(\wt G)$.
Therefore $\IZ/4\cong \langle[S_{28}]\rangle\leq TT(2.G)$.

Now $G$ has a maximal subgroup $H\cong (2^{6}:\PSU_{3}(3)):2$ containing $N$ with $X(H)\cong \IZ/2$.
If $1_{2}\in X(H)$ denotes the non-trivial character of $H$, then $e_{0}\cdot \Ind_{H}^{G}(1_{2})=\chi_{9}+\chi_{25}$. 
For degree reasons $\Gamma_{G}(1_{1})$ affords
$\chi_{9}+\chi_{25}\in\IZ\Irr(G)$, but $\chi_{9}+\chi_{25}$ takes value
$31$ on class $3a$, so that $\Gamma_{G}(1_{1})$ is not endotrivial by
Theorem~\ref{lem:Greenchars}. 
This yields $TT(H)\cong TT(G)\cong \IZ/2$ and $TT(2.G)\cong \IZ/4$.

\subsection{The group $G=\spor{Suz}$}\label{sec:suztors5} \

\textbf{Characteristic 3}. 
In this case $P\cong 3^{5}:3^{2}$, $N\cong 3^{5}:(3^{2}:SD_{16})$ and
$X(N)\cong (\IZ/2)^{2}$. Moreover $G$ has a subgroup $H\cong 3^{5}:M_{11}$
containing $N$, so that $TT(G)\leq TT(H)$ via restriction. Inducing the
three non-trivial linear characters $1_{1},1_{2}, 1_{4}$ of $N$ to $H$
we get: 
$$e_{0}\cdot \Ind_{N}^{H}(1_{1})=\chi_{2}+\chi_{9},\quad e_{0}\cdot \Ind_{N}^{H}(1_{2})=\chi_{5}+\chi_{8},\quad  e_{0}\cdot \Ind_{N}^{H}(1_{4})=\chi_{10}$$
(Note that $H$ has a unique block of full defect.) Since
$\chi_{2},\chi_{5},\chi_{8}, \chi_{9},\chi_{10}\in\Irr(H)$ all have
degree strictly greater than $1$ and class $3a$ in their kernel, it
follows that  the $kH$-Green correspondents
$\Gamma_{H}(1_{1}),\Gamma_{H}(1_{2}),\Gamma_{H}(1_{4})$ cannot be
endotrivial by Theorem~\ref{lem:Greenchars}. A fortiori
$TT(G)=TT(H)=\{[k]\}$. 

\textbf{Characteristic 5}. 
By \cite[Thm. 7.1]{LMS} the character $\chi_{5}\in\Irr(\spor{G})$ reduces modulo $5$ to a
  self-dual endotrivial $kG$-module $S_{1001}$ of dimension~1001. Therefore $\IZ/2\leq
  TT(\spor{Suz})\leq \IZ/2\oplus \IZ/4$ since  $X(N)\cong \IZ/2\oplus\IZ/4$.
  Moreover, $G$ has a maximal subgroup $H\cong \spor{J}_{2}:2$
  containing the normaliser $N\cong 5^{2}:(4\times\fS_{3})$. Similar
  computations as above show that $H$ has two trivial source endotrivial
  modules of order $4$, both affording the character
  $\chi_{18}+\chi_{19}$, with $\chi_{18},\chi_{19}\in\Irr(H)$. 
 The group $H$ also has a trivial source module of order 2 affording the
 character  $\chi_{18}+\chi_{23}$ of degree 666, hence not
 endotrivial. It follows that $\IZ/2\leq TT(G)\leq
 \IZ/4$. Finally a MAGMA computation shows that
 $TT(G)\cong\IZ/2$.    

\subsection{The group $G=\spor{O'N}$ and its three-fold cover in characteristic $7$.}\label{sec:ONtors7}
In this case $P=7_{+}^{1+2}$ and $N\cong7_{+}^{1+2}:(D_{8}\times 3)$ so that $X(N)\cong \IZ/2\oplus \IZ/6$. The principal block is the unique block with full defect.
Now $N$ is contained in a maximal subgroup $H\cong \PSL_{3}(7):2$.  
First, let $1_{a}\in X(H)\cong \IZ/2$ be the non-trivial linear character of
$H$. Then $e_{0}\cdot \Ind_{H}^{G}(1_{a})=\chi_{10}$, which does not
take value  $1$ on $7$-elements. Hence $\Gamma_{G}(1_{a})$ is not
endotrivial by Theorem~\ref{lem:Greenchars}. 
Since $TT(G)$ identifies with a subgroup of $TT(H)$ via restriction, it
remains to show that $TT(H)=X(H)$.  For this we let $1_{i}\in X(N)$,
$1\leq i\leq 12$ denote the non-trivial linear characters of $N$. 
The character  $1_{6}$ has order $3$ and
$\Ind_{N}^{H}(1_{6})=\chi_{7}+\chi_{9}$. The characters $1_{2}, 1_{4}$
have  order $2$, the $kH$-Green correspondent of $1_{4}$ is $1_{a}$ and
$\Ind_{N}^{H}(1_{2})=\chi_{6}+\chi_{19}$. Therefore the characters of
$\Gamma_{H}(1_{6})$ and $\Gamma_{H}(1_{2})$ cannot take value $1$ on
non-trivial $7$-elements, so that $\Gamma_{H}(1_{6})$ and
$\Gamma_{H}(1_{2})$ are not endotrivial  by
Theorem~\ref{lem:Greenchars}. It follows that $TT(H)=X(H)\cong\IZ/2$
and $TT(G)\cong \{[k]\}$.   
Similarly to the situation for $G$, the three-fold cover $3.G=3.\spor{O'N}$ has
a maximal subgroup of shape $3\times H$ containing $\wt N$. We have
$X(\wt N)\cong \IZ/3\oplus\IZ/6$ and $X(3\times H)\cong \IZ/6$. Again since
Green and Brauer correspondences commute, we only need to consider
faithful modules. Let $1_{c},1_{d}\in X(3\times H)$ be the two modules
of order~$6$. Inducing $1_{c}, 1_{d}$ to $3.G$ in GAP, we get that the
characters afforded by  $\Gamma_{G}(1_{c})$ and $\Gamma_{G}(1_{d})$ are
$\chi_{59},\chi_{60}\in\Irr(3.G)$. But both these characters
take the value $15$ on conjugacy class $7a$, so that $\Gamma_{G}(1_{c})$,
$\Gamma_{G}(1_{d})$ cannot be endotrivial by
Theorem~\ref{lem:Greenchars}. It follows that $TT(3.G)$ has no
faithful element of order $3$ and so $TT(3.G)\cong \{[k]\}$.

\subsection{The group $G=\spor{Co}_{3}$ in characteristic
  $3$}\label{sec:co33tors} 
In this case $|P|=3^{7}$. We choose maximal subgroups $H_{1}\cong
3^{5}:(2\times \spor{M}_{11})$ and $H_{2}\cong 3^{1+4}:(4.\fS_{6})$ of $G$
both containing $N$ so that $TT(G)\leq TT(H_{1})$ and $TT(G)\leq
TT(H_{2})$ via restriction. 
Since $|H_{1}:N|=55$ and $|H_{2}:N|=10$, trivial source endotrivial
$kH_{1}$-,\,$kH_{2}$-modules must be $1$-dimensional by
Lemma~\ref{basics}(7), that is $TT(H_{1})\cong X(H_{1})\cong \IZ/2$ and
$TT(H_{2})\cong X(H_{2})\cong (\IZ/2)^{2}$. Then inducing the linear
characters of $N$ in GAP, we find that there is a module $1_{a}\in X(N)$
such that $\Gamma_{H_{1}}(1_{a})\in X(H_{1})$ has order $2$, but
$\Gamma_{H_{2}}(1_{a})$ has  dimension $10$ and thus cannot be
endotrivial. Whence $\Gamma_{G}(1_{a})$ is not endotrivial either and $TT(G)=\{[k]\}$.

\subsection{The group $G=\spor{Co}_{2}$ in characteristic
  $3$}\label{sec:co33tors} 
In this case $|P|=3^{6}$. We choose maximal subgroups $H_{1}\cong
3^{5}:(2\times \spor{M}_{11})$ and $H_{2}\cong 3^{1+4}:(4.\fS_{6})$ of $G$
both containing $N$ so that $TT(G)\leq TT(H_{1})$ and $TT(G)\leq
TT(H_{2})$ via restriction.  
Since $|H_{2}:N|=40$, trivial source endotrivial $kH_{2}$-modules must be 1-dimensional by Lemma~\ref{basics}(7) so that  $TT(H_{2})\cong X(H_{2})\cong \IZ/2$. 
Let $1_{a}\in X(N)$ be the module such that $\Gamma_{H_{2}}\in X(H_{2})$ is 1-dimensional of order $2$. Then 
$$\Ind_{N}^{H_{1}}(1_{a})=\chi_{11}+\chi_{30}+\chi_{31}+\chi_{37}$$
where $\chi_{11},\chi_{30},\chi_{31},\chi_{37}\in\Irr(H_{1})$ all take values greater than $1$ on class $3b$. Thus, by Theorem~\ref{lem:Greenchars}, $\Gamma_{H_{1}}(1_{a})$ cannot be endotrivial, and neither can $\Gamma_{G}(1_{a})$. Hence $TT(G)=\{[k]\}$. 

\subsection{The group $G=\spor{Fi}_{22}$ and its
  covers}\label{sec:fi22tors5} \

\textbf{Characteristic 3}. In this case $P$ has order $3^9$, exponent
$9$ and is an extension of an elementary abelian $3$-group of order
$3^5$ with a Sylow $3$-subgroup of $\PSU_4(2)$, which is itself a wreath
product $3\wr3$. Then $N$ is contained in a subgroup $H\cong \spor
O_7(3)$ of $G$. By \cite[Theorem~6.2]{CMN1}, we have $TT(H)=\{[k]\}$,
and therefore $TT(G)=\{[k]\}$.
  
\textbf{Characteristic 5}.  In this case $P\cong 5^{2}$, $N$ has order $2^{5}\cdot 3\cdot 5^{2}$ and $X(N)\cong \IZ/4$.
  By \cite[Thm. 7.1]{LMS}, the character $\chi_{4}\in\Irr(G)$ reduces modulo $5$ to a
  self-dual endotrivial $kG$-module $S_{1001}$ of dimension~1001, and by \cite[Sec.2]{LM} we have $TT(G)=\left<[S_{1001}]\right>\cong \IZ/2$.
  Furthermore  \cite[Sec.2]{LM} also shows that $TT(2.G)\cong
  \IZ/2\oplus\IZ/2$, while $X(\wt N)\cong \IZ/2\oplus \IZ/4$, and $TT(3.G)\cong
  \IZ/6$, where by \cite[Thm. 7.1]{LMS}, the characters
  $\chi_{115}, \overline{\chi_{115}}\in\Irr(3.\nspor{Fi}_{22})$ reduce
  modulo $5$ to $k[3.G]$-modules $S_{351}$ and $(S_{351})^*$, both
  faithful simple trivial source endotrivial of dimension~$351$. It
  follows that $TT(6.G)\cong \IZ/6\oplus\IZ/2$.  
 
\subsection{The group $G=\spor{HN}$}\label{sec:hntors5} \

\textbf{Characteristic 3}. In this case $|P|=3^{6}$ and we can choose
maximal subgroups $H_{1}\cong 3^{4}:2.(\fA_{4}\times\fA_{4}).4$ and
$H_{2}\cong 3^{1+4}:(4.\fA_{5})$ of $G$, both containing $N$ so that $TT(G)\leq
TT(H_{1})$ and $TT(G)\leq TT(H_{2})$ via restriction. 
Since $|H_{1}:N|=16$ and $|H_{2}:N|=10$, trivial source
endotrivial $kH_{1}$- and $kH_{2}$-modules must be $1$-dimensional by
Lemma~\ref{basics}(7), that is  $TT(G)\leq TT(H_{1})\cong X(H_{1})\cong \IZ/4$ and $TT(G)\leq
TT(H_{2})\cong X(H_{2})\cong\IZ/2$. But if $1_{a}\in X(N)$ is of order $2$
such that $\Gamma_{H_{2}}(1_{a})$ is $1$-dimensional, then
$\Gamma_{H_{1}}(1_{a})$ affords the irreducible character
$\chi_{19}\in\Irr(H_{1})$ of degree $16$, hence is not endotrivial and
neither is $\Gamma_{G}(1_{a})$, whence $TT(G)=\{[k]\}$.  
 
\textbf{Characteristic 5}.
In this case $|P|=5^{6}$. Similarly as above, we choose
maximal subgroups $H_{1}\cong 5^{1+4}:2^{1+4}.5.4$ and $H_{2}\cong
5^{2}.5.5^{2}.(4.\fA_{5})$ of $G$, both containing $N$. 
Since $|H_{1}:N|=16$ and $|H_{2}:N|=6$, we get $TT(H_{1})\cong X(H_{1})\cong \IZ/4$ and
$TT(H_{2})\cong X(H_{2})\cong \IZ/2$. Now if  $1_{a}\in X(N)$ denotes the
$1$-dimensional module such that $\Gamma_{H_{1}}(1_{a})\in X(H_{1})$ has
order $2$, then inducing $1_{a}$ to $H_{2}$ at the level of characters
yields $\Ind_{N}^{H_{2}}(1_{a})=\chi_{7}+\chi_{8}$ where both
$\chi_{7},\chi_{8}\in\Irr(H_{2})$ have degree $3$. Therefore
$\Gamma_{H_{2}}(1_{a})=\Ind_{N}^{H_{2}}(1_{a})$ with dimension $6$ and
is not endotrivial, whence $TT(G)=\{[k]\}$.  

\subsection{The group $G=\spor{Ly}$}\label{sec:lytors5} \

\textbf{Characteristic 3}. In this case $|P|=3^{7}$. 
We choose maximal subgroups $H_{1}\cong 3^{5}:(2\times \spor{M}_{11})$  and $H_{2}\cong 3^{2+4}:2.\fA_{5}.D_{8}$ both containing $N$.
Since $|H_{1}:N|=55$ and $|H_{2}:N|=10$, we get $TT(G)\leq
TT(H_{1})\cong X(H_{1})\cong \IZ/4$ and $TT(G)\leq TT(H_{2})\cong
X(H_{2})\cong \IZ/2$. Thus an argument similar to that used for the simple group $HN$  applies  and it follows that $TT(G)=\{[k]\}$.

\textbf{Characteristic 5}.
In this case $G$ contains a maximal subgroup $H\cong \spor{G}_{2}(5)$, so that
$TT(G)$ identifies with a subgroup of $TT(H)$ via restriction. Since
$TT(H)=\{[k]\}$ by \cite[Thm. A]{CMN1}, we have $TT(G)=\{[k]\}$.

\subsection{The group $G=\spor{Fi}_{23}$ in characteristic
  $5$}\label{sec:fi23tors5}  
 In this case $P$ is elementary abelian of rank~2 and $X(N)\cong \IZ/2\oplus\IZ/4$.
 By \cite[Thm. 7.1]{LMS} the character  $\chi_{9}\in\Irr(G)$ reduces
  modulo $5$ to a self-dual endotrivial $kG$-module  $S_{111826}$ 
  of dimension $111826$. Therefore $\IZ/2\leq TT(G)\leq \IZ/2\oplus\IZ/4$.
Now a maximal subgroup of $G$ is $H=2.\nspor{Fi}_{22}$ and $H$ contains
$N$, so that $TT(G)\leq TT(2.\nspor{Fi}_{22})\cong (\IZ/2)^{2}$ via
restriction. Moreover, if $S_{1001}$ denotes the self-dual simple
endotrivial module of Subsection~\ref{sec:fi22tors5} of dimension~$1001$,  we
compute that $\Gamma_{G}(S_{1001})$ affords $\chi_{5}\in\Irr(G)$ with
degree~$25806\not\equiv 1\pmod{|P|}$. It follows that
$TT(G)\cong\IZ/2$.

\subsection{The group $G=\spor{Co}_{1}$}\label{sec:co1tors5} \

 \textbf{Characteristic 3}. 
 We choose maximal subgroups  $H_{1}\cong
3^{1+4}.2.\PSU_{4}(2).2$ and $H_{2}\cong 3^{6}:2.\nspor{M}_{12}$ both containing $N$ and of index smaller than $|P|=3^{9}$.
Then an argument similar to that used for the simple group $\spor{HN}$  applies  and it follows that $TT(G)=\{[k]\}$. 
 
 \textbf{Characteristic 5}. 
  In this case $|P|=5^{4}$ and $X(N)\cong \IZ/4\oplus\IZ/4$.   Moreover the
  group $G$ has a maximal subgroup $H\cong 5^{1+2}:\GL_{2}(5)$
  containing $N$ and   such that $|H:N|=5$. Thus by
  Lemma~\ref{basics}(7), $TT(G)\leq TT(H)\cong X(H)\cong \IZ/4$. 

\subsection{The group $G=\spor{J}_{4}$}\label{sec:j4tors11}\
 
 \textbf{Characteristic 3}.  In this case $P=3_{+}^{1+2}$, $N=(2\times
 3_{+}^{1+2}:8):2$  and $X(N)\cong \IZ/2\oplus\IZ/4$. We show that
 $TT(G)$ has no element of order $2$, hence is trivial. 
The group $G$ has a maximal subgroup $H\cong 2_{+}^{1+12}.3.M_{22}:2$
containing $N$ and such that $X(H)\cong\IZ/2$. 
Let $1_{1}, 1_{2}, 1_{3}$ be the three linear characters of $N$ of order
$2$. Then $\Gamma_{H}(1_{2})=1_{a}$ the non-trivial linear character of
$H$. 
The irreducible constituents of $\Ind_{N}^{H}(1_{1})$ and
$\Ind_{N}^{H}(1_{3})$ all take a value greater than $1$ on class $3a$,
so that by Theorem~\ref{lem:Greenchars} $\Gamma_{H}(1_{1})$ and
$\Gamma_{H}(1_{3})$ are not endotrivial. 
Finally $e_{0}\cdot \Ind_{H}^{G}(1_{a})=\chi_{40}+\chi_{52}$, where both
$\chi_{40}, \chi_{52}\in\Irr(G)$ take value greater than~$1$ on conjugacy
class $3a$. Hence $\Gamma_{G}(1_{2})$ is not endotrivial either. The
claim follows. 
 
\textbf{Characteristic 11}.  
The group $\spor J_4$ possesses a T.I. Sylow $11$-subgroup
 $P=11_{+}^{1+2}$ by Theorem~\ref{thm:ti}, and we know that $N\cong
 11_{+}^{1+2}:(5\times (2.\fS_{4}))$ (cf \cite{ATLAS}). 
Therefore $X(N)\cong \IZ/10$ and  it
 follows from Lemma~\ref{lem:omnibus}(4),(5) that 
$TT(G)\cong X(N)\cong \IZ/10$.

\subsection{The group $G=\spor{B}$ in characteristic~$5$}\label{sec:btors5} \
The group $\spor{B}$ has a Sylow $5$-subgroup of order $5^{6}$, whose
normaliser $N$ is contained in a maximal subgroup $H\cong \spor{HN:2}$. 
It follows directly from Subsection~\ref{sec:hntors5} that $TT(H)\cong \IZ/2$,
hence $TT(G)\leq \IZ/2$. (We note that in this case $X(N)\cong \IZ/4\oplus\IZ/4$.)


\section{Endotrivial modules for sporadic groups - summary of results}\label{sec:tables}

Table~\ref{tbl:sumup} below summarises the isomorphism types of the
groups $T(G)$ of endotrivial $kG$-modules that we could determine in the
case when the Sylow $p$-subgroups have $p$-rank at least $2$.
An empty entry means that the Sylow $p$-subgroups for the corresponding
prime are either cyclic or $p\nmid |G|$. An isomorphism ``\,$\cong T(G)$\,'' indicates that 
the isomorphism is obtained via inflation. A question mark indicates that only
a partial result for the structure of $TT(G)$ has been
obtained. In these cases we summarise the structure of the
group $X(N)$ in Table~\ref{tab:sporX(N)}, which we obtain using
\cite{GL,GLS3,CTblLib,MAGMA}. The motivation for such computations is
Lemma~\ref{lem:omnibus}(\ref{omni2}), asserting that
$TT(G)\leq X(N)$. A question mark bounded by a group,
e.g. $(?\leq \IZ/2)$ for the group $\spor{B}$ in characteristic~$5$,
indicates that we have found a sharper bound for $TT(G)$ than the one
given by $X(N)$.  

{\scriptsize
\begin{table}[htbp]  \caption{$T(G)$ for covering groups of sporadic groups with $p^{2}$ dividing $|G|$} \label{tbl:sumup}
\vskip -1pt
\[\begin{array}{|c|c|c|c|c|c|c|}
\hline
\hbox{Group $G$ }&p=2&p=3&p=5&p=7&p=11&p=13\\
\hline\hline
\spor M_{11}&\Z\oplus\Z/2&\IZ\oplus(\IZ/2)^{2}&&&&\\
\hline\hline
\spor M_{12}&\Z&\IZ^{3}&&&&\\   \cline{2-7}
2.\nspor M_{12}&\Z&\cong T(M_{12})&&&&\\  
\hline\hline
\spor J_{1}&\Z&&&&&\\
\hline\hline
\spor M_{22}&\Z&\IZ\oplus(\IZ/2)^2 &&&&\\   \cline{2-7}
2.\nspor M_{22}&\Z&\IZ\oplus\IZ/2\oplus \IZ/4&&&&\\  \cline{2-7}
3.\nspor M_{22}&\Z&\IZ&&&&\\   \cline{2-7}
4.\nspor M_{22}&\Z&\cong T(2.\nspor{M}_{22})&&&&\\   \cline{2-7}
6.\nspor M_{22}&\Z&\IZ&&&&\\   \cline{2-7}
12.\nspor M_{22}&\Z&\IZ&&&&\\   \cline{2-7}
\hline\hline
\spor J_{2}&\Z&\IZ\oplus \IZ/2&\IZ\oplus \IZ/2 &&&\\    \cline{2-7}
2.\nspor J_{2}&\Z&\cong T(\spor J_{2})&\cong T(\spor J_{2}) &&&\\
\hline\hline
\spor M_{23}&\Z&\IZ\oplus \IZ/2&&&&\\
\hline\hline
\tw2\spor F_4(2)'&\Z&\IZ^{2}&\IZ\oplus \IZ/6&&&\\
\hline\hline
\spor{HS}&\Z&\IZ\oplus (\IZ/2)^2&\IZ^{2}\oplus \IZ/4&&&\\     \cline{2-7}
2.\nspor{HS}&\Z&\cong T(\spor{HS})&\cong T(\spor{HS})&&&\\
\hline\hline
\spor J_{3}&\Z&\IZ\oplus (?\leq\IZ/2)&&&&\\    \cline{2-7}
3.\nspor J_{3}&\Z\oplus \Z/3 &\IZ&&&&\\
\hline\hline
\spor M_{24}&\Z&\IZ^{2}&&&&\\
\hline\hline
\spor{McL}&\Z&\IZ&\IZ\oplus \IZ/8&&&\\     \cline{2-7}
3.\nspor{McL}&\Z&\IZ&\IZ\oplus \IZ/24&&&\\
\hline\hline
\spor{He}&\Z&\IZ^{2}&\IZ\oplus \IZ/3 &\IZ^{3}\oplus ?&&\\
\hline\hline
\spor{Ru}&\Z&\IZ\oplus \IZ/2&\IZ^{2}\oplus ?&&&\\     \cline{2-7}
2.\nspor{Ru}&\Z&\IZ\oplus \IZ/4&\cong T(\spor{Ru})&&&\\
\hline\hline
\spor{Suz}&\Z&\IZ&\IZ\oplus \IZ/2&&&\\     \cline{2-7}
2.\nspor{Suz}&\Z&\cong T(\spor{Suz})&\cong T(\spor{Suz})&&&\\     \cline{2-7}
3.\nspor{Suz}&\Z& \IZ &\cong T(\spor{Suz})&&&\\     \cline{2-7}
6.\nspor{Suz}&\Z& \IZ &\cong T(\spor{Suz})&&&\\  
\hline\hline
\spor{O'N}&\Z&\IZ\oplus (?\leq \IZ/2)&&\IZ^{3}&&\\      \cline{2-7}
3.\nspor{O'N}&\Z&\IZ&&\IZ^{3}&&\\
\hline\hline
\spor{Co}_3&\Z&\IZ&\IZ\oplus ?&&&\\
\hline\hline
\spor{Co}_2&\Z&\IZ&\IZ\oplus ?&&&\\
\hline\hline
\spor{Fi}_{22}&\Z&\IZ&\IZ\oplus \IZ/2&&&\\     \cline{2-7}
2.\nspor{Fi}_{22}&\Z&\cong T(\spor{Fi}_{22})&\IZ\oplus (\IZ/2)^{2}&&&\\     \cline{2-7}
3.\nspor{Fi}_{22}&\Z&\IZ&\IZ\oplus \IZ/6&&&\\  \cline{2-7}
6.\nspor{Fi}_{22}&\Z&\IZ&\IZ\oplus \IZ/6\oplus\IZ/2&&&\\
\hline\hline
\spor F_5 = \spor{HN}&\Z&\IZ&\IZ&&&\\
\hline\hline
\spor{Ly}&\Z& \IZ &\IZ&&&\\
\hline\hline
\spor{Th}&\Z&\IZ\oplus ?&\IZ\oplus ?&\IZ\oplus ?&&\\
\hline\hline
\spor{Fi}_{23}&\Z&\IZ\oplus ?&\IZ\oplus \IZ/2&&&\\
\hline\hline
\spor{Co}_1&\Z&\IZ&\IZ^{2}\oplus( ?\leq\IZ/4)&\IZ\oplus ?&&\\      \cline{2-7}
2.\nspor{Co}_1&\Z&\cong T(\spor{Co}_1)&\cong T(\spor{Co}_1)&\cong T(\spor{Co}_1)&&\\
\hline\hline
\spor J_4&\Z&\IZ&&&\IZ\oplus \IZ/10&\\
\hline\hline
\spor{Fi}'_{24} &\Z&\IZ\oplus ?&\IZ\oplus ?&\IZ^{3}\oplus ?&&\\     \cline{2-7}
3.\nspor{Fi}'_{24} &\Z&\IZ&\cong T(\spor{Fi}'_{24})&\cong T(\spor{Fi}'_{24})&&\\
\hline\hline
\spor F_2 = \spor B&\Z&\IZ\oplus ?&\IZ\oplus(?\leq \IZ/2)&\IZ\oplus ?&&\\     \cline{2-7}
2.\nspor F_2 = 2.\nspor B&\Z&\cong T(\spor{B})&\cong T(\spor{B})&\cong T(\spor{B})&&\\
\hline\hline
\spor F_1 = \spor M&\Z&\IZ\oplus ?&\IZ\oplus?&\IZ\oplus?&\IZ\oplus ?&\IZ^{2}\oplus ?\\
\hline
\end{array}\]
\end{table}
}


Finally, we assume that $G$ is quasi-simple with $G/Z(G)$ sporadic simple and
a Sylow $p$-subgroup $P$ of $G$ is cyclic. Then $P\cong C_{p}$ and $p\neq 2$ (see
\cite{ATLAS}). 
We use Theorem~\ref{thm:cyclic} with, in our case, $H=N_{G}(P)$ and
${e=|N_{G}(P):C_{G}(P)|}$ is the inertial index of the principal block of
$G$. The results come down to a routine computation using the structure of
$H$ which can easily be computed using \cite{GLS3,GAP4}, and from 
which we get $X(H)$ and hence $T(G)$. We present them in Table~\ref{tbl:cyclic}. 

To compute $T(G)$ in this case, we recall that $T(G)\cong T(H)$, and the short exact sequence
$$\xymatrix{0\ar[r]&X(H)\ar[r]&T(H)\ar[r]^{\Res^H_P}&T(P)\ar[r]&0}$$
splits if and only if the class $[\Omega^{2}(k_{H})]$ is a square in $X(H)$. 
In particular the sequence splits if $e$ is odd
in which case $T(G)\cong X(H)\oplus\IZ/2$. Moreover $T(G)$ always
contains a subgroup $\left<\Omega\right>\cong\IZ/2e$. If $|X(H)|=e$,
then the principal block is the unique block containing
endotrivial modules so that $T(G)=\langle\Omega\rangle\cong\IZ/2e$. If
$e$ is even, we use the distribution of endotrivial modules into blocks 
as described in \cite[Table 7]{LMS}. Moreover, if $m.G$ is a covering group of $G$ and $p\,\nmid\,m$, then $T(G)$
identifies with a subgroup of $T(m.G)$ via inflation. This is enough to conclude in all cases. \\


\begin{table}[htbp]
\caption{The structure of $X(N)$ for the missing cases in Table \ref{tbl:sumup}}\label{tab:sporX(N)}
\[\begin{array}{|c|c|c||c|c|c||c|c|c|}
\hline
     G& p&  X(N) &   G& p& X(N)& G & p & X(N)   \cr
\hline
  \spor{He}  & 7& \IZ/6&  			    		\spor{Th}&7  &  \IZ/6  &			   	\spor{B}  & 7 &\IZ/2\oplus\IZ/6 \cr
  \spor{Ru} & 5&\IZ/2\oplus\IZ/4 &           		\spor{Fi}_{23} &3 &(\IZ/2)^{3}  &   	         \spor{M}  & 3 & (\IZ/2)^3                   \cr
  \spor{O'N}   & 3& \IZ/2&         		     	      	\spor{Co}_{1} &7 & \IZ/3\oplus\IZ/3 & 	\spor{M}  &  5& \IZ/2\oplus(\IZ/4)^{2}     \cr
\spor{Co}_{3} & 5& \IZ/2\oplus\IZ/4&     	    	\spor{Fi}_{24}' &3 &(\IZ/2)^{3}&	      		\spor{M}  & 7 &  \IZ/2\oplus\IZ/6		   \cr
\spor{Co}_{2}  & 5&\IZ/4 &     			     	\spor{Fi}_{24}'  &5  & \IZ/4&      	         		 \spor{M}  &11  & \IZ/5     \cr
\spor{Th} &3 &  (\IZ/2)^2  & 				\spor{Fi}_{24}' &7 &  \IZ/2\oplus\IZ/6 &         \spor{M} & 13& \IZ/12   \cr
\spor{Th} & 5&  \IZ/4  &					\spor{B}  &3  & (\IZ/2)^3   &			& & \cr	
\hline
\end{array}\]
\end{table}


{\scriptsize
\begin{table}[htbp]  \caption{$T(G)$ for covering groups of sporadic groups with cyclic Sylow $p$-subgroups}\label{tbl:cyclic}
\vskip -3pt
\[\begin{array}{|r|r|r|c|c||r|r|r|c|c|}
\hline
 G& p& X(H)& e &T(G) & G& p&  X(H) &e &T(G)  \\
\hline \hline
\tw2F_4(2)'& 13&\IZ/6 & 6 &\IZ/12 & \spor{He}&      17&  \IZ/8& 8&  \IZ/16  \\   \cline{1-10} \cline{1-10}

\spor{M}_{11}&    5& \IZ/4 &4  &   \IZ/8& \spor{Ru}&       7&  \IZ/6& 6&   \IZ/12   \\ \cline{2-5}
\spor{M}_{11}&   11&  \IZ/5& 5&     \IZ/10 &  2.\nspor{Ru}&       7&   \IZ/6&6 &   \cong T(\spor{Ru})  \\ \cline{1-5} \cline{1-5} \cline{7-10}

\spor{M}_{12}&    5&  \IZ/2\oplus \IZ/4& 4&\textcolor{black}{\IZ/2\oplus \IZ/8}& \spor{Ru}&      13& \IZ/12& 12&  \IZ/24   \\
2.\nspor{M}_{12}& 5&  \IZ/2\oplus \IZ/4& 4& \cong T(\spor{M}_{12})  & 2.\nspor{Ru}&       13& \IZ/12 &12 &   \cong T(\spor{Ru})    \\ \cline{2-5} \cline{7-10}
\spor{M}_{12}&   11&   \IZ/5& 5&   \IZ/10 & \spor{Ru}&      29& \IZ/14& 14&  \IZ/28 \\
2.\nspor{M}_{12}& 11&  \IZ/10& 5&   \IZ/2\oplus \IZ/10& 2.\nspor{Ru}&    29& \IZ/28& 14&   \textcolor{black}{\IZ/2\oplus\IZ/28}  \\  \cline{1-10} \cline{1-10}

\spor{J}_{1}&     3&  \IZ/2\oplus \IZ/2& 2 &  \textcolor{black}{\IZ/2\oplus \IZ/4} & \spor{Suz}&      7&  \IZ/6& 6&    \IZ/12    \\   \cline{2-5}
\spor{J}_{1}&     5&  \IZ/2\oplus \IZ/2& 2 &   \textcolor{black}{\IZ/2\oplus \IZ/4}  & 2.\nspor{Suz}&    7&   \IZ/6 & 6  &      \cong T(\spor{Suz})    \\  \cline{2-5}
\spor{J}_{1}&     7&   \IZ/6&  6&    \IZ/12 &3.\nspor{Suz}&    7&    \IZ/3\oplus  \IZ/6& 6 &    \textcolor{black}{\IZ/3\oplus \IZ/12}   \\     \cline{2-5}
\spor{J}_{1}&    11&  \IZ/10&  10&     \IZ/20 &  6.\nspor{Suz}&    7& \IZ/3\oplus  \IZ/6 & 6 &     \cong T(3.\nspor{Suz})  \\     \cline{2-5} \cline{7-10}
\spor{J}_{1}&    19&   \IZ/6&  6&     \IZ/12 & \spor{Suz}&     11& \IZ/10& 10&    \IZ/20   \\    \cline{1-5} \cline{1-5}

\spor{M}_{22}&    5& \IZ/4&  4&   \IZ/8     & 2.\nspor{Suz}&   11& \IZ/20& 10&   \textcolor{black}{\IZ/2\oplus \IZ/20} \\
2.\nspor{M}_{22}&  5& \IZ/2\oplus \IZ/4&  4&   \IZ/2\oplus \IZ/8  & 3.\nspor{Suz}&   11& \IZ/30& 10&  \IZ/60        \\
3.\nspor{M}_{22}&  5& \IZ/12&  4&  \IZ/24  & 6.\nspor{Suz}&   11& \IZ/60& 10&     \textcolor{black}{\IZ/2\oplus  \IZ/60}    \\  \cline{7-10}
4.\nspor{M}_{22}&  5& \IZ/2\oplus \IZ/8&  4& \textcolor{black}{\IZ/4\oplus \IZ/8}  & \spor{Suz}&     13&  \IZ/6& 6&     \IZ/12      \\
6.\nspor{M}_{22}&  5& \IZ/2\oplus \IZ/12&  4&   \textcolor{black}{\IZ/2\oplus \IZ/24} & 2.\nspor{Suz}&   13& \IZ/12& 6&     \textcolor{black}{\IZ/2\oplus \IZ/12}         \\
12.\nspor{M}_{22}& 5&  \IZ/2\oplus \IZ/24& 4&  \textcolor{black}{\IZ/4\oplus \IZ/24} & 3.\nspor{Suz}&   13&  \IZ/3\oplus  \IZ/6& 6 &     \textcolor{black}{\IZ/3\oplus \IZ/12}     \\  \cline{2-5}
\spor{M}_{22}&    7&   \IZ/3&  3&     \IZ/6 & 6.\nspor{Suz}&   13& \IZ/3\oplus\IZ/12& 6&     \textcolor{black}{\IZ/6\oplus  \IZ/12}   \\  \cline{6-10} \cline{6-10}
2.\nspor{M}_{22}&  7&   \IZ/6 & 3&   \IZ/2\oplus  \IZ/6 & \spor{ON}&       5&  \IZ/2\oplus\IZ/4& 4&     \textcolor{black}{ \IZ/2\oplus\IZ/8 }  \\  
3.\nspor{M}_{22}&  7&   \IZ/3\oplus\IZ/3&  3&    \IZ/3\oplus\IZ/6  & 3.\nspor{ON}&       5&   \IZ/2\oplus\IZ/4 & 4 &         \cong T(\spor{ON})     \\  \cline{7-10}
4.\nspor{M}_{22}&  7& \IZ/12&  3&       \IZ/2\oplus  \IZ/12 &  \spor{ON}&      11& \IZ/10& 10&        \IZ/20 \\ 
6.\nspor{M}_{22}&  7& \IZ/3\oplus\IZ/6&  3&  \IZ/6\oplus\IZ/6  & 3.\nspor{ON}&    11& \IZ/30& 10 &   \IZ/60   \\   \cline{7-10}
12.\nspor{M}_{22}& 7& \IZ/3\oplus\IZ/12& 3&    \IZ/6\oplus\IZ/12 &\spor{ON}&      19&  \IZ/6& 6&        \IZ/12   \\  \cline{2-5}
\spor{M}_{22}&   11&  \IZ/5&  5&        \IZ/10 &3.\nspor{ON}&    19& \IZ/3\oplus\IZ/6& 6&    \textcolor{black}{\IZ/3\oplus \IZ/12}    \\ \cline{7-10}
2.\nspor{M}_{22}& 11&  \IZ/10&  5&    \IZ/2\oplus  \IZ/10 & \spor{ON}&      31& \IZ/15& 15&    \IZ/30  \\
3.\nspor{M}_{22}& 11&  \IZ/15&  5&       \IZ/30  & 3.\nspor{ON}&    31& \IZ/3\oplus\IZ/15& 15&    \IZ/3\oplus\IZ/30      \\  \cline{6-10} \cline{6-10}
4.\nspor{M}_{22}& 11& \IZ/20&  5&      \IZ/2\oplus \IZ/20   &\spor{Co}_{3}&     7& \IZ/2\oplus\IZ/6& 6&    \textcolor{black}{\IZ/2\oplus \IZ/12}  \\    \cline{7-10}
6.\nspor{M}_{22}& 11& \IZ/30&  5&        \IZ/2\oplus  \IZ/30 & \spor{Co}_{3}&    11& \IZ/10& 5&     \IZ/2\oplus \IZ/10 \\     \cline{7-10}
12.\nspor{M}_{22}& 11& \IZ/60& 5&         \IZ/2\oplus  \IZ/60 & \spor{Co}_{3}&    23& \IZ/11& 11&       \IZ/22\\   \cline{1-10} \cline{1-10}

\spor{J}_{2}&    7&    \IZ/6&  6&      \IZ/12  & \spor{Co}_{2}&     7&  \IZ/2\oplus\IZ/6& 6&  \textcolor{black}{\IZ/2\oplus \IZ/12}   \\    \cline{7-10}
2.\nspor{J}_{2}&  7&  \IZ/12&  6&     \textcolor{black}{\IZ/2\oplus \IZ/12} &  \spor{Co}_{2}&    11& \IZ/10& 10&      \IZ/20    \\  \cline{1-5} \cline{1-5}   \cline{7-10}

\spor{M}_{23}&   5&  \IZ/4& 4&       \IZ/8 & \spor{Co}_{2}&    23& \IZ/11& 11&        \IZ/22  \\   \cline{2-5}  \cline{6-10}\cline{6-10}
\spor{M}_{23}&   7&  \IZ/6& 3&   \IZ/2\oplus\IZ/6 & \spor{Fi}_{22}&    7& \IZ/2\oplus\IZ/6& 6&   \textcolor{black}{\IZ/2\oplus \IZ/12}    \\   \cline{2-5}
\spor{M}_{23}&  11&   \IZ/5& 5&     \IZ/10 & 2.\nspor{Fi}_{22}&  7&\IZ/2\oplus \IZ/2\oplus\IZ/6& 6&    \textcolor{black}{\IZ/2\oplus \IZ/2\oplus \IZ/12}       \\   \cline{2-5}
\spor{M}_{23}&  23&  \IZ/11& 11&      \IZ/22  & 3.\nspor{Fi}_{22}&  7&  \IZ/6\oplus\IZ/6& 6&   \textcolor{black}{\IZ/6\oplus\IZ/12}   \\    \cline{1-5} \cline{1-5} 

\spor{HS}&       7&  \IZ/6 & 6&    \IZ/12  & 6.\nspor{Fi}_{22}&  7&  \IZ/2\oplus \IZ/6\oplus\IZ/6&6&   \textcolor{black}{\IZ/2\oplus \IZ/6\oplus\IZ/12}       \\    \cline{7-10}
2.\nspor{HS}&     7& \IZ/12& 6&    \textcolor{black}{\IZ/2\oplus \IZ/12}   & \spor{Fi}_{22}&   11& \IZ/10& 5&    \IZ/2\oplus\IZ/10      \\  \cline{2-5}
\spor{HS}&      11&   \IZ/5& 5&     \IZ/10  & 2.\nspor{Fi}_{22}& 11&  \IZ/2\oplus\IZ/10& 5&      \IZ/2\oplus\IZ/2\oplus\IZ/10   \\
2.\nspor{HS}&    11& \IZ/10 & 5&  \IZ/2\oplus  \IZ/10 &3.\nspor{Fi}_{22}& 11& \IZ/30& 5&     \IZ/2\oplus \IZ/30    \\     \cline{1-5} \cline{1-5} 

\spor{J}_{3}&    5&  \IZ/2\oplus\IZ/2& 2&     \textcolor{black}{\IZ/2\oplus \IZ/4} &  6.\nspor{Fi}_{22}& 11& \IZ/2\oplus\IZ/30& 5&   \IZ/2\oplus \IZ/2\oplus\IZ/30     \\    \cline{7-10}
3.\nspor{J}_{3}&  5& \IZ/2\oplus\IZ/6& 2&    \textcolor{black}{\IZ/2\oplus \IZ/12}  &\spor{Fi}_{22}&   13&  \IZ/6& 6&   \IZ/12     \\   \cline{2-5}
\spor{J}_{3}&   17&  \IZ/8& 8&       \IZ/16  &2.\nspor{Fi}_{22}& 13& \IZ/2\oplus\IZ/6& 6  & \textcolor{black}{\IZ/2\oplus \IZ/12}   \\
3.\nspor{J}_{3}& 17& \IZ/24& 8&        \IZ/48 & 3.\nspor{Fi}_{22}& 13&  \IZ/3\oplus\IZ/6&  6&     \textcolor{black}{\IZ/3\oplus \IZ/12}   \\  \cline{2-5}
\spor{J}_{3}&   19&  \IZ/9& 9&       \IZ/18 & 6.\nspor{Fi}_{22}& 13&\IZ/6\oplus\IZ/6& 6 &   \textcolor{black}{\IZ/6\oplus \IZ/12}   \\   \cline{6-10} \cline{6-10}
3.\nspor{J}_{3}& 19& \IZ/3\oplus\IZ/9& 9&      \IZ/3\oplus\IZ/18 &  \spor{HN}&         7&  \IZ/6& 6&   \IZ/12      \\     \cline{1-5} \cline{1-5}  \cline{7-10}

\spor{M}_{24}&   5&  \IZ/4& 4&     \IZ/8   &\spor{HN}&        11& \IZ/2\oplus\IZ/10 & 10&    \textcolor{black}{\IZ/2\oplus \IZ/20}     \\ \cline{2-5}  \cline{7-10}
\spor{M}_{24}&   7&  \IZ/6& 3&     \IZ/2\oplus\IZ/6  &\spor{HN}&        19&  \IZ/9& 9&        \IZ/18  \\   \cline{2-5}    \cline{6-10} \cline{6-10}
\spor{M}_{24}&  11& \IZ/10& 10&       \IZ/20 & \spor{Ly}&         7&  \IZ/6& 6&       \IZ/12  \\  \cline{2-5}   \cline{7-10}
\spor{M}_{24}&  23& \IZ/11& 11&       \IZ/22  & \spor{Ly}&        11& \IZ/10& 5&  \IZ/2\oplus\IZ/10  \\   \cline{1-5} \cline{1-5} \cline{7-10}

\spor{McL}&      7&   \IZ/6& 3&     \IZ/2\oplus  \IZ/6 &\spor{Ly}&        31&  \IZ/6& 6&          \IZ/12   \\     \cline{7-10}
3.\nspor{McL}&    7& \IZ/3\oplus \IZ/6& 3 &     \IZ/6\oplus \IZ/6  &\spor{Ly}&        37& \IZ/18& 18&         \IZ/36  \\  \cline{2-5}   \cline{7-10}
\spor{McL}&     11&  \IZ/5& 5&     \IZ/10   & \spor{Ly}&        67& \IZ/22& 22&         \IZ/44  \\   \cline{6-10} \cline{6-10}
3.\nspor{McL}&   11&  \IZ/15& 5&  \IZ/30  & \spor{Th}&        13& \IZ/12& 12&      \IZ/24    \\
\hline
\end{array}\]
\end{table}

\begin{table}[htbp]
\vskip -1pt
\[\begin{array}{|r|r|r|c|c||r|r|r|c|c|}
\hline
 G& p& X(H)& e &T(G) & G& p&  X(H) &e &T(G)  \\
\hline \hline

\spor{Th}&        19& \IZ/18& 18&      \IZ/36  &                                                                                                      	 \spor{Fi}_{24}'&   23&  \IZ/11& 11&          \IZ/22     \\   \cline{2-5}
\spor{Th}&        31& \IZ/15& 15&      \IZ/30 & 												 3.\nspor{Fi}_{24}'& 23& \IZ/33& 11&    \IZ/66 \\   \cline{1-5} \cline{1-5} \cline{7-10}

\spor{Fi}_{23}&     7& \IZ/2\oplus\IZ/6& 6&   \textcolor{black}{\IZ/2\oplus \IZ/12}& 					 \spor{Fi}_{24}'&   29& \IZ/14& 14&       \IZ/28   \\ \cline{2-5}  \cline{7-10}
\spor{Fi}_{23}&    11& \IZ/2\oplus\IZ/10& 10&   \textcolor{black}{\IZ/2\oplus \IZ/20}& 				 3.\nspor{Fi}_{24}'& 29& \IZ/42& 14&        \IZ/84   \\ \cline{2-5}    \cline{6-10} \cline{6-10}
\spor{Fi}_{23}&    13& \IZ/2\oplus\IZ/6& 6&   \textcolor{black}{\IZ/2\oplus \IZ/12}& 					 \spor{B}&          11& \IZ/2\oplus\IZ/10& 10&     \textcolor{black}{\IZ/2\oplus \IZ/20}    \\ \cline{2-5}
\spor{Fi}_{23}&    17& \IZ/16& 16&    \IZ/32&											 2.\nspor{B}&       11&   \IZ/2\oplus\IZ/10& 10 &    \cong T(\spor{B})     \\ \cline{2-5}    \cline{7-10}
\spor{Fi}_{23}&    23& \IZ/11& 11&    \IZ/22& 											 \spor{B}&          13& \IZ/2\oplus\IZ/12& 12&     \textcolor{black}{\IZ/2\oplus \IZ/24}  \\   \cline{1-5} \cline{1-5}

\spor{Co}_{1}&       11&  \IZ/2\oplus\IZ/10& 10&   \textcolor{black}{\IZ/2\oplus \IZ/20}& 				 2.\nspor{B}&       13& \IZ/2\oplus\IZ/12  & 12 &   \cong T(\spor{B})   \\  \cline{7-10}
2.\nspor{Co}_{1}& 11&  \IZ/2\oplus\IZ/10& 10&  \cong T(\spor{Co}_{1})&								\spor{B}&          17& \IZ/2\oplus\IZ/16& 16&       \textcolor{black}{\IZ/2\oplus \IZ/32}   \\ \cline{2-5}  
\spor{Co}_{1}&       13& \IZ/12& 12&     \IZ/24&											 2.\nspor{B}&       17&   \IZ/2\oplus\IZ/16 & 16 &    \cong T(\spor{B})     \\   \cline{7-10}
2.\nspor{Co}_{1}& 13&\IZ/12 &12 &  \cong T(\spor{Co}_{1})&										\spor{B}&          19& \IZ/2\oplus\IZ/18& 18&      \textcolor{black}{\IZ/2\oplus \IZ/36}   \\ \cline{2-5}
\spor{Co}_{1}&       23& \IZ/11& 11&     \IZ/22&											 2.\nspor{B}&       19&  \IZ/2\oplus\IZ/18& 18 &    \cong T(\spor{B})    \\   \cline{7-10}
2.\nspor{Co}_{1}&     23& \IZ/22& 11&   \IZ/2\oplus\IZ/22 &									 \spor{B}&          23& \IZ/22& 11&   \IZ/2\oplus\IZ/22    \\   \cline{1-5} \cline{1-5}

\spor{J}_4&         5&  \IZ/4& 4&    \IZ/8&													2.\nspor{B}&       23& \IZ/2\oplus\IZ/22& 11&     \IZ/2\oplus \IZ/2\oplus\IZ/22    \\ \cline{2-5}   \cline{7-10}  \cline{7-10}
\spor{J}_4&         7&  \IZ/6& 3&    \IZ/2\oplus\IZ/6 &   										\spor{B}&          31& \IZ/15& 15&           \IZ/30    \\ \cline{2-5}
\spor{J}_4&        23& \IZ/22& 22&        \IZ/44& 											2.\nspor{B}&       31&  \IZ/30& 15&            \IZ/2\oplus  \IZ/30 \\ \cline{2-5}   \cline{7-10}
\spor{J}_4&        29& \IZ/28& 28&        \IZ/56& 											\spor{B}&          47&  \IZ/23& 23&           \IZ/46 \\ \cline{2-5}
\spor{J}_4&        31& \IZ/10& 10&        \IZ/20& 											2.\nspor{B}&       47&    \IZ/46& 23&          \IZ/2\oplus  \IZ/46 \\ \cline{2-5}   \cline{6-10} \cline{6-10}
\spor{J}_4&        37& \IZ/12& 12&        \IZ/24&											\spor{M}&          17& \IZ/16& 16&       \IZ/32    \\ \cline{2-5}   \cline{7-10}
\spor{J}_4&        43& \IZ/14& 14&        \IZ/28&											\spor{M}&          19& \IZ/18& 18&            \IZ/36   \\ \cline{1-5} \cline{1-5}   \cline{7-10}

\spor{Fi}_{24}'&   11& \IZ/10& 10&           \IZ/20&											\spor{M}&          23&  \IZ/22& 11&       \IZ/2\oplus  \IZ/22  \\   \cline{7-10}
3.\nspor{Fi}_{24}'&  11&  \IZ/30& 10&     \IZ/60&  							\spor{M}&          29&  \IZ/28& 28&        \IZ/56    \\   \cline{2-5}   \cline{7-10}
\spor{Fi}_{24}'&   13& \IZ/2\oplus\IZ/12& 12&       \textcolor{black}{\IZ/2\oplus \IZ/24}& 				\spor{M}&          31& \IZ/30& 15&         \IZ/2\oplus \IZ/30   \\     \cline{7-10}
3.\nspor{Fi}_{24}'& 13& \IZ/2\oplus\IZ/12 & 12 &      \cong T(\spor{Fi}_{24}')& 							\spor{M}&          41&  \IZ/40& 40&        \IZ/80 \\  \cline{2-5}   \cline{7-10}
\spor{Fi}_{24}'&   17&  \IZ/16& 16&           \IZ/32     & 										\spor{M}&          47&  \IZ/46& 23&     \IZ/2\oplus \IZ/46      \\   \cline{7-10}
3.\nspor{Fi}_{24}'& 17&  \IZ/48& 16&          \IZ/96  &  						\spor{M}&          59&  \IZ/29& 29&      \IZ/58  \\   \cline{7-10}
& & & & & 																	\spor{M}&          71& \IZ/35& 35&      \IZ/70    \\   \cline{7-10}
 \hline
\end{array}\]
\end{table}
}


\bibliographystyle{alpha}
\bibliography{biblio.bib}


\end{document}